\documentclass[final]{siamltex}

\usepackage{setspace,graphicx}
\usepackage{subfig}
\usepackage{todonotes}
\usepackage{amsfonts}
\usepackage{verbatim}
\usepackage{bm}
\usepackage{empheq}

\usepackage{color}

\usepackage{mathtools}
\usepackage{array}
\usepackage{booktabs}
\usepackage{amssymb}
\usepackage{amsmath}
\usepackage{cases}
\usepackage{color}
\usepackage{placeins}
\usepackage{stmaryrd}
\newtheorem{remark}[theorem]{Remark}

\newcommand{\calE}{\mathcal{E}}
\newcommand{\calT}{\mathcal{T}}

\newcommand{\norm}[1]{\| #1 \|}
\newcommand{\normDG}[1]{\| #1 \|_{\textrm{DG}}}
\newcommand{\normast}[1]{| #1 |_{*,h}}
\newcommand{\normh}[1]{\| #1 \|_{1,h}}
\newcommand{\abs}[1]{\left| #1 \right|}
\newcommand{\dA}{\ \textrm{dA}}
\newcommand{\ds}{\ \textrm{ds}}
\newcommand{\dx}{\ \textrm{dx}}
\newcommand{\dAh}{\ \textrm{dA}_{\textrm{h}}}
\newcommand{\dAhk}{\ \textrm{dA}_{\textrm{hk}}}
\newcommand{\dsh}{\ \textrm{ds}_{\textrm{h}}}
\newcommand{\dshk}{\ \textrm{ds}_{\textrm{hk}}}
\newcommand{\half}{\frac{1}{2}}

\newcommand{\WT}[1]{\widetilde{#1}}
\newcommand{\WH}[1]{\widehat{#1}}
\newcommand{\ADG}{\mathcal{A}_{h}^{k}}
\newcommand{\AC}{\mathcal{A}}

\makeatletter
\def\@avgsp{\mathchoice{{}\mkern-6mu}{{}\mkern-6mu}{{}\mkern-6mu}{}}
\def\llaverage{\{\@avgsp\{}
\def\rraverage{\}\@avgsp\}}
\makeatother

\author{Paola F. Antonietti\thanks{MOX---Modeling and Scientific Computing, Dipartimento di Matematica, Politecnico di Milano, Milano, Italy, I-20133. Mail: paola.antonietti@polimi.it; simone.stangalino@polimi.it; marco.verani@polimi.it.} \and Andreas Dedner\thanks{Mathematics Institute and Centre for Scientific Computing, University of Warwick, Coventry CV4 7AL, UK. Mail: A.S.Dedner@warwick.ac.uk; P.Madhavan@warwick.ac.uk; Bjorn.Stinner@warwick.ac.uk.} \and Pravin Madhavan\footnotemark[2], \\ Simone Stangalino\footnotemark[1] \and Bj\"{o}rn Stinner\footnotemark[2] \and Marco Verani\footnotemark[1]}
\title{High order discontinuous Galerkin methods\\ on surfaces}

\begin{document}

\maketitle

\begin{abstract}
We derive and analyze high order discontinuous Galerkin methods for   
second-order elliptic problems on implicitely defined surfaces in $\mathbb{R}^{3}$. This is done by carefully adapting the unified discontinuous Galerkin framework of \cite{arnold2002unified} on a triangulated surface approximating the smooth surface. We prove optimal error estimates in both a (mesh dependent) energy and $L^2$ norms. 
\end{abstract}



\pagestyle{myheadings}
\thispagestyle{plain}
\markboth{HIGH ORDER DG METHODS ON SURFACES}{ANTONIETTI, DEDNER, MADHAVAN, STANGALINO, STINNER AND VERANI}

\section{Introduction}
Partial differential equations (PDEs) on manifolds have become an active area of research in recent years due to the fact that, in many applications, mathematical models have to be formulated not on a flat Euclidean domain but on a curved surface. For example, they arise naturally in fluid dynamics (e.g.,~surface active agents on the interface between two fluids, \cite{JamLow04}) and material science (e.g.,~diffusion of species along grain boundaries, \cite{DecEllSty01}) but have also emerged in other areas as image processing and cell biology (e.g.,~cell motility involving processes on the cell membrane, \cite{neilson2011modelling} or phase separation on biomembranes, \cite{EllSti10}). 

Finite element methods (FEMs) for elliptic problems and their error analysis have been successfully applied to problems on surfaces via the intrinsic approach in \cite{dziuk1988finite}. This approach has subsequently been extended to parabolic problems \cite{dziuk2007surface} as well as evolving surfaces \cite{dziuk2007finite}. The literature on the application of FEM to various surface PDEs is now quite extensive, a review of which can be found in \cite{dziuk2013finite}. High order error estimates, which require high order surface approximations, have been derived in \cite{demlow2009higher} for the Laplace-Beltrami operator. However, there are a number of situations where conforming FEMs may not be the appropriate numerical method, for instance, problems which lead to steep gradients or even discontinuities in the solution. Such issues can arise for problems posed on surfaces, as in \cite{sokolov2012numerical} where the authors analyse a model for bacteria/cell aggregation. Without an appropriate stabilisation mechanism artificially added to the surface FEMs scheme, the solution can exhibit a spurious oscillatory behaviour which, in the context of the above problem, leads to negative densities of on-surface living cells. 

Given the ease with which one can perform hp-adaptivity using high order discontinuous Galerkin (DG) methods and its in-built stabilisation mechanisms for dealing with advection dominated problems and solution blow-ups, it is natural to extend the DG framework for PDEs posed on surfaces. DG methods have first been extended to surfaces in \cite{dedner2013surfaces}, where an interior penalty (IP) method for a linear second-order elliptic problem was introduced and optimal a priori error estimates in the $L^{2}$ and energy norms for piecewise linear ansatz functions and surface approximations were derived. A posteriori error estimates have then been derived for this surface IP method in \cite{dedner2013aposteriori}.  A continuous/discontinuous Galerkin method for a fourth order elliptic PDE on surfaces is considered in \cite{larsson2013continuous}; \cite{ju2009finite}, \cite{lenz2011convergent} and \cite{GieMue_prep} have also derived a priori error bounds for finite volume methods on (evolving) surfaces via the intrinsic approach.


In this paper, we consider a second-order elliptic equation on a compact smooth connected and oriented surface $\Gamma \subset \mathbb{R}^{3}$ and, following the unified framework of \cite{arnold2002unified} based on the so called flux formulation and the high order surface approximation approach considered in \cite{demlow2009higher}, derive the high order DG formulation on a piecewise polynomial approximation $\Gamma_{h}^{k}$ of $\Gamma$, where $k \geq 1$ is the polynomial order of the approximation. The derivation requires a suitable integration by parts formula which holds on discrete surfaces; this differs from the conventional one used in the planar case. Then, by choosing the numerical fluxes appropriately, we derive ``surface'' counter-parts of the various planar DG bilinear forms discussed in \cite{arnold2002unified}.  

We then perform a unified a priori error analysis of the surface $DG$ methods and derive estimates in the $L^2$ and energy norms by relating $\Gamma_{h}^{k}$ to $\Gamma$ via the surface lifting operator introduced in \cite{dziuk1988finite}. The estimates are a generalisation of the a priori error estimates derived in \cite{dedner2013surfaces} for the surface interior penalty (IP) method, which restricted the analysis to the linear case. The geometric error terms arising when approximating the surface involve those present for the surface FEM method given in \cite{demlow2009higher} as well as additional terms arising from the DG methods. The latter are shown to scale with the same order as the former and hence we obtain optimal convergence rates as long as the surface approximation order and the $DG$ space order coincide.
     

The paper is organised in the following way. Section 2 presents the model problem which we investigate, following the approach taken in \cite{dziuk1988finite}. In Section 3 we present a unified framework for high order DG methods on surfaces and derive the bilinear forms corresponding to each of the classical DG methods outlined in \cite{arnold2002unified}. In Section 4 we describe the technical estimates needed to prove the convergence of the surface DG methods, which is then reported in Section 5. 

\section{Model problem}
\label{sec:NotationAndSetting}
The notation in this section closely follows that used in \cite{dziuk1988finite}.
Let $\Gamma$ be a compact smooth connected and oriented surface in $\mathbb{R}^{3}$, with $\partial \Gamma = \emptyset$, for simplicity, and let $d(\cdot)$ denote the signed distance function to $\Gamma$ which we assume to be well-defined in a sufficiently thin open tube $U$ around $\Gamma$. The orientation of $\Gamma$ is set by taking the normal $\nu$ of $\Gamma$ to be in the direction of increasing $d(\cdot)$, i.e.,
\[\nu(\xi)  = \nabla d(\xi),\ \xi \in \Gamma. \]
We denote by $\pi(\cdot)$ the projection onto $\Gamma$, i.e.,~$\pi:U \rightarrow \Gamma$ is given by
\begin{equation}\label{eq:uniquePoint}
\pi(x)  = x - d(x)\nu(x) \quad \mbox{where } \nu(x) =\nu(\pi(x)).
\end{equation}
In the following, we assume that there is a one-to-one relation between points $x \in U$ and points $\xi=\pi(x) \in \Gamma$. In particular, (\ref{eq:uniquePoint}) is invertible in $U$. 
We denote by
\[ P(\xi) = I - \nu(\xi)\otimes \nu(\xi),\ \xi\in \Gamma, \]
the projection onto the tangent space $T_{\xi}\Gamma$ on $\Gamma$ at a point $\xi \in \Gamma$, where $\otimes$ denotes the usual tensor product.

\begin{remark}
It is easy to see that
\begin{equation}
\label{Da}
\nabla \pi=P-dH,
\end{equation}
where $H = \nabla^{2}d$ \cite[Lemma 3]{dziuk1988finite}.
\end{remark}

For any function $\eta$ defined in an open subset of $U$ containing $\Gamma$ we define its \emph{tangential gradient} on $\Gamma$ by
\[ \nabla_{\Gamma}\eta  = \nabla \eta - \left(\nabla \eta\cdot \nu \right) \nu = P\nabla \eta,\]
and the \emph{Laplace-Beltrami} operator by 
\[ \Delta_{\Gamma} \eta  = \nabla_{\Gamma}\cdot (\nabla_{\Gamma} \eta).
\]

For an integer $m\geq 0$, we define the surface Sobolev space 
$ H^{m}(\Gamma)  = \{u \in L^{2}(\Gamma) : D^{\alpha}u \in L^{2}(\Gamma)\ \forall |\alpha| \leq m \}$. For $s=0$ we write $L^2(\Gamma)$ instead of $H^0(\Gamma)$.
We endow the Sobolev space with the standard seminorm and norm
\[ 
|u|_{H^{m}(\Gamma)}  = \left(\sum_{|\alpha|=m} \norm{D^{\alpha}u}_{L^{2}(\Gamma)}^{2}\right)^{1/2}, \quad 
\norm{u}_{H^{m}(\Gamma)}  = \left(\sum_{k=0}^m |u|_{H^{k}(\Gamma)}^{2}\right)^{1/2},
\]
respectively, cf \cite{wlokapartial}. 
Throughout the paper, we write $x \lesssim y$ to signify $x < C  y$, where $C$ is a generic positive constant whose value, possibly different at any occurrence, does not depend on the meshsize. 
Moreover, we use $x\sim y$ to state the equivalence between $x$ and $y$, i.e., $C_1 y \leq x \leq C_2 y$, for $C_1,\ C_2$ independent of the meshsize.

%

Let $f \in L^{2}(\Gamma)$ be a given function, we consider the following
model problem: Find $u \in H^{1}(\Gamma)$ such that
\begin{align}\label{eq:weakH1}
&& \int_\Gamma \nabla_{\Gamma}u\cdot \nabla_{\Gamma}v + u v\dA=\int_\Gamma fv\dA
&&\forall v\in H^1(\Gamma). 
\end{align}
Throughout the paper, we assume that $u \in H^s(\Gamma)$, $s\geq 2$. Existence and uniqueness of such a solution is shown in \cite{aubin1982nonlinear}.

\section{High order DG approximation} \label{sec:ApproximationAndProperties}
We now follow the high order surface approximation framework introduced in \cite{demlow2009higher}. We begin by approximating the smooth surface $\Gamma$ by a polyhedral surface $\Gamma_{h} \subset U$ composed of planar triangles $\{\WT{K}_{h}\}$ whose vertices lie on $\Gamma$, and denote by $\WT{\calT}_{h}$ the associated regular, conforming triangulation of $\Gamma_{h}$, i.e., $\Gamma_{h} = \bigcup_{\WT{K}_{h} \in \WT{\calT}_{h}} \WT{K}_{h}$.

We next describe a family $\Gamma^k_h$ of polynomial approximations to $\Gamma$ of degree $k \geq 1$ (with the convention that 
$\Gamma_h^1 =\Gamma_h$). For a given element $\WT{K}_{h} \in\WT{\calT}_{h}$, let $\{\phi^k_i\}_{1\leq i\leq n_k}$
be the Lagrange basis functions of degree $k$ defined on $\WT{K}_{h}$ corresponding to a set nodal points $x_1 , ..., x_{n_k}$.
For $x\in\WT{K}_{h}$, we define the discrete projection $\pi_k:\Gamma_h\rightarrow U$ as
$$\pi_k(x) =\sum_{j=1}^{n_k}\pi(x_j )\phi_j^k(x).$$
By constructing $\pi_k$ elementwise  we obtain a continuous piecewise
polynomial map on $\Gamma_h$. We then define the corresponding discrete surface
$\Gamma^k_h  = \{\pi_k(x) : x \in\Gamma_h \}$ and the corresponding regular, conforming triangulation $\WH{\calT}_h = \{\pi_k(\WT{K}_{h})\}_{\WT{K}_{h} \in \WT{\calT}_{h}}$. We denote by $\WH{\calE}_h$ the set of all (codimension one) intersections $\WH{e}_h$ of elements in $\WH{\calT}_h$, i.e., $\WH{e}_{h}=\WH{K}_{h}^+\cap\WH{K}_{h}^-$, 
for some elements $\WH{K}_{h}^{\pm}\in \WH{\calT}_h$. For any $\WH{e}_{h}\in \WH{\calE}_h$, the conormal $n_h^+$ to a point $ x \in \WH{e}_{h}$ is the unique unit vector that belongs to $T_x\WH{K}_h^{+}$ and that satisfies
$$n_h^+(x)\cdot (x-y)\geq 0\ \ \  \forall y\in \WH{K}_{h}^+\cap B_\epsilon(x),$$ where $B_\epsilon(x)$ is the ball centered in $x$ with (small enough) radius $\epsilon>0$.
Analogously, one can define the conormal $n_h^-$ on $\WH{e}_{h}$ by exchanging $\WH{K}_{h}^+$ with $\WH{K}_{h}^-$. Notice that with the above definition $n_h^+\neq -n_h^-$, in general (see Figure \ref{img:Conormals}). Finally, we denote by $\nu_{h}$ the outward unit normal to $\Gamma_{h}^{k}$ and define for each $\WH{K}_{h} \in \WH{\calT}_h$ the discrete projection $P_{h}$ onto the tangential space of $\Gamma_{h}^{k}$ by 
\[ P_{h}(x) = I - \nu_{h}(x) \otimes \nu_{h}(x),\ x \in \WH{K}_{h}, \]
so that, for $v_{h}$ defined on $\Gamma_{h}^{k}$,
\[\nabla_{\Gamma_{h}^{k}}v_{h} = P_{h}\nabla v_{h}.\]  

\begin{figure}[!hbt]
\begin{center}
\includegraphics[scale=1.4]{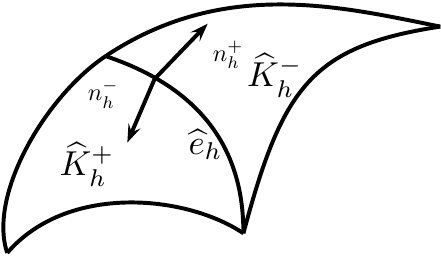}
\caption{Example of two elements in $\WH{\calT}_h$ and their respective conormals on the common edge $\WH{e}_{h}$.}\label{img:Conormals}
\end{center}
\end{figure}
Let $K \subset \mathbb{R}^{2}$ be the (flat) reference element and let $F_{\WH{K}_{h}} : K \rightarrow \WH{K}_{h} \subset  \mathbb{R}^{3}$ for $\WH{K}_{h} \in \WH{\calT}_h$. We define the DG space associated to $\Gamma_h^k$ by 
\begin{align*}
\WH{S}_{hk} =\{\WH{\chi}\in L^2(\Gamma_h^k):\WH{\chi}|_{\WH{K}_{h}}= \chi \circ F_{\WH{K}_{h}}^{-1} \text{ for some } \chi \in \mathbb{P}^{k}(K)\ \ \ \forall \WH{K}_{h}\in\WH{\calT}_{h}\}.
\end{align*}
For $v_h\in\WH{S}_{hk}$ we adopt the convention that $v_h^{\pm}$ is the trace of $v_h$ on $\WH{e}_{h}=\WH{K}_{h}^+\cap\WH{K}_{h}^-$ taken within the interior of $\WH{K}_{h}^{\pm}$, respectively. In addition, we define the vector-valued function space
\[\WH{\Sigma}_{hk} = \{\WH{\tau}\in [L^2(\Gamma_h^k)]^{3}:\WH{\tau}|_{\WH{K}_{h}}= \nabla F_{\WH{K}_{h}}^{-T}\left( \tau \circ F_{\WH{K}_{h}}^{-1} \right) \text{ for some } \tau \in [\mathbb{P}^{k}(K)]^{2}\ \ \ \forall \WH{K}_h\in\WH{\calT}_h \}. \]
Here, $\nabla F_{\WH{K}_{h}}^{-1}$ refers to the (left) \emph{pseudo-inverse} of $\nabla F_{\WH{K}_{h}}$, i.e., $$\nabla F_{\WH{K}_{h}}^{-1} = \left(\nabla F_{\WH{K}_{h}}^{T}\nabla F_{\WH{K}_{h}}\right)^{-1}\nabla F_{\WH{K}_{h}}^{T}.$$
Note that $P_{h} \nabla F_{\WH{K}_{h}}^{-T} = \nabla F_{\WH{K}_{h}}^{-T}$, i.e., $\WH{\tau} \in \WH{\Sigma}_{hk} \Rightarrow \WH{\tau} \in T_{x}\Gamma_{h}^{k}$ almost everywhere. This result straightforwardly implies that $\eta \in \WH{S}_{hk} \Rightarrow \nabla_{\Gamma_{h}^{k}} \eta \in \WH{\Sigma}_{hk}$.


\subsection{Primal formulation}
Rewriting (\ref{eq:weakH1}) as a first order system of equations and following the lines of \cite{arnold2002unified}, we wish to find $(u_{h}, \sigma_{h}) \in \WH{S}_{hk} \times \WH{\Sigma}_{hk}$ such that

\begin{align*}
&\int_{\WH{K}_{h}} \sigma_{h}\cdot \tau_{h} \dAhk = -\int_{\WH{K}_{h}} u_{h} \nabla_{\Gamma_{h}^{k}} \cdot \tau_{h}\dAhk + \int_{\partial \WH{K}_{h}}\WH{u}\ \tau_{h} \cdot n_{\WH{K}_{h}} \dshk,\\
&\int_{\WH{K}_{h}} \sigma_{h}\cdot \nabla_{\Gamma_{h}^{k}} v_{h} + u_{h} v_{h} \dAhk = \int_{\WH{K}_{h}}f_{h} v_{h}\dAhk + \int_{\partial \WH{K}_{h}}\WH{\sigma} \cdot n_{\WH{K}_{h}}\ v_{h}\dshk,
\end{align*}
for all $\tau_{h} \in \WH{\Sigma}_{hk}$, $v_{h} \in \WH{S}_{hk}$ and where the discrete right-hand side $f_h\in L^2(\Gamma_h^k)$ will be related to $f$ in Section \ref{sec:SurfaceLifting}. Here $\WH{u} = \WH{u}(u_{h})$ and $\WH{\sigma} = \WH{\sigma}(u_{h},\sigma_{h}(u_{h}))$ are the so called numerical fluxes which determine the inter-element behaviour of the solution and will be prescribed later on. 

In order to deal with these terms, we need to introduce the following trace operators:
\begin{align*}
q\in L^2(\Gamma): \{q\}=\frac{1}{2}(q^++q^-)&, \ [q]=q^+-q^- \ \text{on } \WH{e}_{h}\in\WH{\calE}_{h}^{k},\\
\phi\in [L^2(\Gamma)]^3: \{\phi;n_{h}\}=\frac{1}{2}(\phi^+\cdot n_{h}^+-\phi^-\cdot n_{h}^-)&, \ [\phi;n_{h}]=\phi^+\cdot n_{h}^++\phi^-\cdot n_{h}^- \ \text{on } \WH{e}_{h}\in\WH{\calE}_{h}^{k}.
\end{align*}
We now state and prove a useful formula which holds for functions in 
\[ H^{1}(\WH{\mathcal{T}}_{h})=\{ v|_{\WH{K}_{h}} \in H^{1}(\WH{K}_{h})\ :\ \forall \WH{K}_{h} \in \WH{\mathcal{T}}_{h} \}. \]
\begin{lemma}\label{eq:intByParts}
Let $\phi \in [H^{1}(\WH{\mathcal{T}}_{h})]^{3}$ and $\psi \in H^{1}(\WH{\mathcal{T}}_{h})$. Then we have that
\begin{align*}
\sum_{\WH{K}_{h} \in \WH{\calT}_{h}}\int_{\partial \WH{K}_{h}}\psi \phi \cdot n_{\WH{K}_{h}} \dshk=\sum_{\WH{e}_{h} \in \WH{\calE}_{h}}\int_{\WH{e}_{h}}[\phi;n_{h}]\{\psi\} + \{\phi;n_{h}\}[\psi] \dshk.
\end{align*}
\end{lemma}
\begin{proof}
The result follows straightforwardly by noting that
\begin{align*}
\sum_{\WH{K}_{h} \in \WH{\calT}_{h}}\int_{\partial \WH{K}_{h}}\psi \phi \cdot n_{\WH{K}_{h}} \dshk &=  \sum_{\WH{e}_{h} \in \WH{\calE}_{h}}\int_{\WH{e}_{h}}[\psi \phi; n_{h}] \dshk\\
& = \sum_{\WH{e}_{h} \in \WH{\calE}_{h}}\int_{\WH{e}_{h}}[\phi;n_{h}]\{\psi\} + \{\phi;n_{h}\}[\psi] \dshk.
\end{align*}
\end{proof}

\begin{remark}
The formula in Lemma \ref{eq:intByParts} is a generalisation to surfaces of the classical (planar) formula given in (2.1) of \cite{arnold82}. 
\end{remark}

Applying the above lemma, summing over all elements and proceeding in a similar fashion to \cite{arnold2002unified}, we obtain
\begin{align}
\sum_{\WH{K}_{h} \in \WH{\calT}_{h}}\int_{\WH{K}_{h}} \sigma_{h} \cdot \tau_{h}\ \dAhk =& \sum_{\WH{K}_{h} \in \WH{\calT}_{h}}\int_{\WH{K}_{h}}\nabla_{\Gamma_{h}^{k}}u_{h}\cdot \tau_{h}\ \dAhk \notag \\
& + \sum_{\WH{e}_{h} \in \WH{\calE}_{h}}\int_{\WH{e}_{h}}[\WH{u}-u_{h}] \{\tau_{h}; n_{h}\} + \{\WH{u}-u_{h}\}[\tau_{h}; n_{h}]\ \dshk, \label{eq:primalDerivation1}
\end{align}
\begin{align}
\sum_{\WH{K}_{h} \in \WH{\calT}_{h}}\int_{\WH{K}_{h}} \sigma_{h} \cdot \nabla_{\Gamma_{h}^{k}}v_{h} + u_{h} v_{h}\ \dAhk = &\sum_{\WH{K}_{h} \in \WH{\calT}_{h}}\int_{\WH{K}_{h}}f_{h}v_{h}\ \dAhk \notag \\
&+ \sum_{\WH{e}_{h} \in \WH{\calE}_{h}}\int_{\WH{e}_{h}}\Big(\{ \WH{\sigma};n_{h}\}[v_{h}] + [\WH{\sigma};n_{h}]\{ v_{h} \}\Big)\ \dshk, \label{eq:primalDerivation2}
\end{align}
for every $\tau_{h} \in \WH{\Sigma}_{hk}$ and $v_{h} \in \WH{S}_{hk}$.\\

We now introduce the DG lifting operators $r_{\WH{e}_{h}} : L^{2}(\WH{\mathcal{E}}_{h}) \rightarrow \WH{\Sigma}_{hk}$ and ${l_{\WH{e}_{h}}:L^{2}(\WH{\mathcal{E}}_{h}) \rightarrow \WH{\Sigma}_{hk}}$ which satisfy
$$\int_{\Gamma_{h}^{k}}r_{\WH{e}_{h}}(\phi)\cdot \tau_h \dAhk = - \int_{\WH{e}_{h}} \phi\{\tau_h;n_{h}\}\dshk \ \ \ \ \forall \tau_h\in\WH{\Sigma}_{hk},$$
$$\int_{\Gamma_{h}^{k}}l_{\WH{e}_{h}}(q)\cdot \tau_h \dAhk = -\int_{\WH{e}_{h}} q[\tau_h;n_{h}]\dshk\ \ \ \ \forall \tau_h\in\WH{\Sigma}_{hk},$$
and $r_{h} : L^{2}(\WH{\calE}_{h}) \rightarrow \WH{\Sigma}_{hk}$ and $l_{h}:L^{2}(\WH{\calE}_{h}) \rightarrow \WH{\Sigma}_{hk}$, given by
$$r_h(\phi) = \sum_{\WH{e}_{h} \in \WH{\calE}_{h}} r_{\WH{e}_{h}}(\phi), \qquad l_h(\phi) = \sum_{\WH{e}_{h} \in \WH{\calE}_{h}} l_{\WH{e}_{h}}(\phi).$$

Using these operators, we can write $\sigma_{h}$ solely in terms of $u_{h}$. Indeed, on each element $\WH{K}_{h} \in \WH{\calT}_{h}$ we obtain from (\ref{eq:primalDerivation1}) that
\begin{align} \label{eq:sigmah}
\sigma_{h} = \sigma_{h}(u_{h})= \nabla_{\Gamma_{h}^{k}}u_{h} - r_{h}([\WH{u}(u_{h}) - u_{h}]) - l_{h}(\{\WH{u}(u_{h}) - u_{h}\}).
\end{align}
Note that (\ref{eq:sigmah}) does in fact imply that $\sigma_{h} \in \WH{\Sigma}_{hk}$ as $ \nabla_{\Gamma_{h}^{k}}u_{h} \in \WH{\Sigma}_{hk}$ and $r_{h}, l_{h} \in \WH{\Sigma}_{hk}$ by construction. Taking $\tau_{h} = \nabla_{\Gamma_{h}^{k}}v_{h}$ in (\ref{eq:primalDerivation1}), substituting the resulting expression into (\ref{eq:primalDerivation2}) and using (\ref{eq:sigmah}), we obtain the primal formulation: find $(u_{h}, \sigma_{h}) \in \WH{S}_{hk} \times \WH{\Sigma}_{hk}$ such that
\begin{align}\label{eq:PrimalFormulation}
\ADG(u_{h},v_{h}) = \sum_{\WH{K}_{h} \in \WH{\calT}_{h}}\int_{\WH{K}_{h}} f_{h}v_{h}\ \dAhk\ \  \forall v_{h} \in \WH{S}_{hk}, 
\end{align}  
where
\begin{align}\label{eq:PrimalForm}
\ADG(u_{h},v_{h})&= \sum_{\WH{K}_{h} \in \WH{\calT}_{h}}\int_{\WH{K}_{h}}\nabla_{\Gamma_{h}^{k}}u_{h}\cdot \nabla_{\Gamma_{h}^{k}}v_{h} + u_{h}v_{h}\ \dAhk \notag \\
&+ \sum_{\WH{e}_{h} \in \WH{\calE}_{h}}\int_{\WH{e}_{h}}([\WH{u}-u_{h}] \{\nabla_{\Gamma_{h}^{k}}v_{h}; n_{h}\} - \{\WH{\sigma}; n_{h}\}[v_{h}])\ \dshk \notag \\ 
&+ \sum_{\WH{e}_{h} \in \WH{\calE}_{h}}\int_{\WH{e}_{h}}(\{\WH{u}-u_{h}\}[\nabla_{\Gamma_{h}^{k}}v_{h}; n_{h}] - [\WH{\sigma}; n_{h}]\{v_{h}\})\ \dshk.
\end{align}

\subsection{Examples of surface DG methods}\label{DGfluxes}
For the following methods we introduce the penalization coefficients $\eta_{\WH{e}_{h}}$ and $\beta_{\WH{e}_{h}}$ defined as
\begin{equation}\label{penalityParameters}
\eta_{\WH{e}_{h}}=\alpha, \quad \beta_{\WH{e}_{h}}=\alpha k^{2} h_{\WH{e}_{h}}^{-1},
\end{equation}
where $\alpha>0$ is a parameter at our disposal. 

\subsubsection{Surface Bassi-Rebay method}
To derive the surface Bassi-Rebay method, based on \cite{bassi1997high}, we choose
\begin{alignat*}{3}
\WH{u}^{+}&=\{u_h\},\quad &\WH{u}^{-}&=\{u_h\},\\
\WH{\sigma}^+&= \{\sigma_h;n_{h}\}n_{h}^+,\quad &\WH{\sigma}^-&=-\{\sigma_h;n_{h}\}n_{h}^-.
\end{alignat*}

From (\ref{eq:sigmah}) we obtain $\sigma_h=\nabla_{\Gamma_{h}^{k}}u_h+r_h([u_h])$ and

\begin{align*}
\sum_{\WH{e}_{h}\in\WH{\calE}_{h}} \int_{\WH{e}_{h}} &\{\WH{\sigma};n_{h}\}[v_h]\dshk \\
&=\sum_{\WH{e}_{h}\in\WH{\calE}_{h}} \int_{\WH{e}_{h}} \{\sigma_h;n_{h}\}[v_h]\dshk\\
&=\sum_{\WH{e}_{h}\in\WH{\calE}_{h}} \int_{\WH{e}_{h}} \{\nabla_{\Gamma_{h}^{k}}u_h;n_{h}\}[v_h]\dshk + \sum_{\WH{e}_{h}\in\WH{\calE}_{h}} \int_{\WH{e}_{h}} \{r_h([u_h]);n_{h}\}[v_h]\dshk\\
&=\sum_{\WH{e}_{h}\in\WH{\calE}_{h}} \int_{\WH{e}_{h}} \{\nabla_{\Gamma_{h}^{k}}u_h;n_{h}\}[v_h]\dshk - \sum_{\WH{K}_{h}\in\WH{\calT}_{h}} \int_{\WH{K}_{h}} r_h([u_h])\cdot r_{h}([v_h])\dAhk.
\end{align*}

Therefore
\begin{align}\label{eq:BassiRebayGammahForm}
\ADG(u_h,v_h)=&\sum_{\WH{K}_{h}\in\WH{\calT}_{h}} \int_{\WH{K}_{h}} \bigg(\nabla_{\Gamma_{h}^{k}}u_h\cdot \nabla_{\Gamma_{h}^{k}}v_h + u_{h} v_{h} +r_h([u_h])\cdot r_h([v_h])\bigg)\dAhk \notag \\
& -\sum_{\WH{e}_{h}\in\WH{\calE}_{h}} \int_{\WH{e}_{h}} \bigg(\{\nabla_{\Gamma_{h}^{k}}u_h;n_{h}\}[v_h] + \{\nabla_{\Gamma_{h}^{k}}v_h;n_{h}\}[u_h]\bigg) \dshk.
\end{align}

\subsubsection{Surface Brezzi et al. method}
For the surface Brezzi et al. method, based on \cite{brezzi1999discontinuous}, we choose
\begin{alignat*}{3}
\WH{u}^{+}&=\{u_h\},\quad &\WH{u}^{-}&=\{u_h\},\\
\WH{\sigma}^+&=\{\sigma_h+\eta_{\WH{e}_{h}}r_{\WH{e}_{h}}([u_h]);n_{h}\}n_{h}^+,\quad &\WH{\sigma}^-&=-\{\sigma_h+\eta_{\WH{e}_{h}}r_{\WH{e}_{h}}([u_h]);n_{h}\}n_{h}^-,
\end{alignat*}

The method is similar to the Bassi-Rebay one with an additional term. Indeed,

\begin{align*}
\sum_{\WH{e}_{h}\in\WH{\calE}_{h}} &\int_{\WH{e}_{h}} \{\WH{\sigma};n_{h}\}[v_h]\dshk\\
&=\sum_{\WH{e}_{h}\in\WH{\calE}_{h}} \int_{\WH{e}_{h}} \{\sigma_h+\eta_{\WH{e}_{h}}r_{\WH{e}_{h}}([u_h]);n_{h}\}[v_h]\dshk\\
&=\sum_{\WH{e}_{h}\in\WH{\calE}_{h}} \int_{\WH{e}_{h}} \{\nabla_{\Gamma_{h}^{k}}u_h;n_{h}\}[v_h]\dshk + \sum_{\WH{e}_{h}\in\WH{\calE}_{h}} \int_{\WH{e}_{h}} \{r_h([u_h])+\eta_{\WH{e}_{h}}r_{\WH{e}_{h}}([u_h]);n_{h}\}[v_h]\dshk\\
&=\sum_{\WH{e}_{h}\in\WH{\calE}_{h}} \int_{\WH{e}_{h}} \{\nabla_{\Gamma_{h}^{k}}u_h;n_{h}\}[v_h]\dshk - \sum_{\WH{K}_{h}\in\WH{\calT}_{h}} \int_{\WH{K}_{h}} r_h([u_h])\cdot r_h([v_h])\dAhk \\
& \qquad - \sum_{\WH{K}_{h}\in\WH{\calT}_{h}} \int_{\WH{K}_{h}} \eta_{\WH{e}_{h}}r_{\WH{e}_{h}}([u_h])\cdot r_{\WH{e}_{h}}([v_h])\dAhk.
\end{align*}
Then
\begin{align}\label{eq:BrezziEtAlGammahForm}
\ADG(u_h,v_h)&=+\sum_{\WH{K}_{h}\in\WH{\calT}_{h}} \int_{\WH{K}_{h}} \nabla_{\Gamma_{h}^{k}}u_h\cdot \nabla_{\Gamma_{h}^{k}}v_h + u_{h} v_{h}\dAhk\notag \\
&-\sum_{\WH{e}_{h}\in\WH{\calE}_{h}} \int_{\WH{e}_{h}} \{\nabla_{\Gamma_{h}^{k}}u_h;n_{h}\}[v_h] + \{\nabla_{\Gamma_{h}^{k}}v_h;n_{h}\}[u_h]\dshk\notag \\
&+\sum_{\WH{K}_{h}\in\WH{\calT}_{h}} \int_{\WH{K}_{h}} r_h([u_h])\cdot r_h([v_h])+\eta_{\WH{e}_{h}} r_{\WH{e}_{h}}([u_h])\cdot r_{\WH{e}_{h}}([v_h])\dAhk.
\end{align}

\subsubsection{Surface IP method}\label{SIPGMethod}
To derive the surface IP method, based on \cite{douglas1976interior,arnold82}, we choose the numerical fluxes $\WH{u}$ and $\WH{\sigma}$ as follows:
\begin{alignat*}{3}
\WH{u}^{+}&=\{u_h\},\quad &\WH{u}^{-}&=\{u_h\},\\
\WH{\sigma}^{+} &=\bigg( \{ \nabla_{\Gamma_{h}^{k}}u_{h}; n_{h}\} - \beta_{\WH{e}_{h}}[u_{h}]\bigg) n_{h}^{+},\quad &\WH{\sigma}^{-} &= -\bigg( \{ \nabla_{\Gamma_{h}^{k}}u_{h}; n_{h}\} - \beta_{\WH{e}_{h}}[u_{h}]\bigg) n_{h}^{-}.
\end{alignat*}
Substituting them into (\ref{eq:PrimalForm}), we obtain

\begin{align}\label{eq:InteriorPenaltyGammahForm}
\ADG(u_{h},v_{h}) &= \sum_{\WH{K}_{h} \in \WH{\calT}_{h}}\int_{\WH{K}_{h}}\nabla_{\Gamma_{h}^{k}}u_{h}\cdot \nabla_{\Gamma_{h}^{k}}v_{h} + u_{h} v_{h}\ \dAhk + \sum_{\WH{e}_{h} \in \WH{\calE}_{h}}\int_{\WH{e}_{h}} \beta_{\WH{e}_{h}}[u_{h}][v_{h}]\ \dshk \notag \\
& - \sum_{\WH{e}_{h} \in \WH{\calE}_{h}}\int_{\WH{e}_{h}} \big( [u_{h}] \{\nabla_{\Gamma_{h}^{k}}v_{h}; n_{h} \} + [v_{h}]\{\nabla_{\Gamma_{h}^{k}}u_{h}; n_{h}\} \big) \dshk
\end{align}
which is exactly the surface IP method considered in \cite{dedner2013surfaces}.

\subsubsection{Surface NIPG method}
For the surface NIPG method, based on \cite{riviere1999improved} (or equivalently the Baumann-Oden method in \cite{baumann1998discontinuous} with $\beta_{\WH{e}_{h}}=0$), we choose
\begin{alignat*}{3}
\WH{u}^{+}&=\{u_h\}+[u_h],\quad &\WH{u}^{-}&=\{u_h\}-[u_h],\\
\WH{\sigma}^+&=\bigg( \{\nabla_{\Gamma_{h}^{k}}u_h;n_{h}\} - \beta_{\WH{e}_{h}}[u_h]\bigg) n_{h}^+,\quad &\WH{\sigma}^-&=-\bigg( \{\nabla_{\Gamma_{h}^{k}}u_h;n_{h}\} - \beta_{\WH{e}_{h}}[u_h]\bigg) n_{h}^-.
\end{alignat*}
We see that $\{\WH{u}-u_h\}=0$, $[\WH{u}-u_h]=[u_h]$ and $[\WH{\sigma};n_{h}]=0$. We may derive the surface NIPG bilinear form in a similar way as for the surface IP method.

\subsubsection{Surface IIPG method}
For the surface IIPG method, based on \cite{Dawson20042565}, we choose the numerical fluxes $\WH{u}$ and $\WH{\sigma}$ as follows:
\begin{alignat*}{3}
\WH{u}^{+}&=u_h^{+},\quad &\WH{u}^{-}&=u_h^{-},\\
\WH{\sigma}^+&=\bigg( \{\nabla_{\Gamma_{h}^{k}}u_h;n_{h}\} - \beta_{\WH{e}_{h}}[u_h]\bigg) n_{h}^+,\quad &\WH{\sigma}^-&=-\bigg( \{\nabla_{\Gamma_{h}^{k}}u_h;n_{h}\} - \beta_{\WH{e}_{h}}[u_h]\bigg) n_{h}^-.
\end{alignat*}
We see that $\{\WH{u}-u_h\}=0$, $[\WH{u}-u_h]=0$ and $[\WH{\sigma};n_{h}]=0$. Here again, we may derive the surface IIPG bilinear form in like manner as for the surface IP method.

\subsubsection{Surface Bassi et al. method}
For the surface Bassi et al. method, based on \cite{bassi1997high2}, we choose
\begin{alignat*}{3}
\WH{u}^{+}&=\{u_h\},\quad &\WH{u}^{-}&=\{u_h\},\\
\WH{\sigma}^+&=\bigg( \{\nabla_{\Gamma_{h}^{k}}u_h+\eta_{\WH{e}_{h}}r_{\WH{e}_{h}}([u_h]);n_{h}\}\bigg) n_{h}^+,\quad &\WH{\sigma}^-&=-\bigg( \{\nabla_{\Gamma_{h}^{k}}u_h+\eta_{\WH{e}_{h}}r_{\WH{e}_{h}}([u_h]);n_{h}\}\bigg) n_{h}^-.
\end{alignat*}
The resulting bilinear surface form can be easily obtained using the contributes of the surface IP and surface Brezzi et al. bilinear forms.

\subsubsection{Surface LDG method}
Finally for the surface LDG method, based on \cite{cockburn1998local}, the numerical fluxes are chosen as follows:
\begin{align*}
&\WH{u}^{+}=\{u_h\}-\beta\cdot n_{h}^+[u_h],\quad \WH{u}^{-}=\{u_h\}-\beta\cdot n_{h}^+[u_h],\\
&\WH{\sigma}^+=\bigg(\{\sigma_h;n_{h}\}-\beta_{\WH{e}_{h}}[u_h] +\beta\cdot n_{h}^+[\sigma_h;n_{h}] \bigg)n_{h}^+,\\
&\WH{\sigma}^-=-\bigg(\{\sigma_h;n_{h}\}-\beta_{\WH{e}_{h}}[u_h] +\beta\cdot n_{h}^+[\sigma_h;n_{h}] \bigg)n_{h}^-,
\end{align*}
where $\beta \in [L^{\infty}(\Gamma_{h}^{k})]^{3}$ is a (possibly null) constant on each edge $\WH{e}_{h} \in \WH{\mathcal{E}}_{h}$. We see that $\{\WH{u}-u_h\}=-\beta\cdot n_{h}^+[u_h]$ and $[\WH{u}-u_h]=-[u_h]$. So, from (\ref{eq:sigmah}), we obtain:
\begin{align*}
\WH{\sigma}^+=\bigg(&\{\nabla_{\Gamma_{h}^{k}}u_h;n_{h}\} + \{r_h([u_h]);n_{h}\} + \{\beta\cdot n_{h}^+l_h([u_h]);n_{h}\}-\beta_{\WH{e}_{h}}[u_h] \\
&+\beta\cdot n_{h}^+\Big( [\nabla_{\Gamma_{h}^{k}}u_h;n_{h}] + [r_h([u_h]);n_{h}]+ [\beta\cdot n_{h}^+l_h([u_h]);n_{h}] \Big) \bigg)n_{h}^+,
\end{align*}

and in a similar way $\WH{\sigma}^-$. Then

\begin{align*}
\sum_{\WH{e}_{h}\in\WH{\calE}_{h}^{k}} &\int_{\WH{e}_{h}} \{\WH{\sigma};n_{h}\}[v_h]\dshk\\
&=\sum_{\WH{e}_{h}\in\WH{\calE}_{h}^{k}} \int_{\WH{e}_{h}} \Big(\{\nabla_{\Gamma_{h}^{k}}u_h;n_{h}\}[v_h]+[\nabla_{\Gamma_{h}^{k}}u_h;n_{h}]\beta\cdot n_{h}^+[v_h]  -\beta_{\WH{e}_{h}}[u_h][v_h] \Big)\dshk\\
& -\sum_{\WH{K}_{h}\in\WH{\calT}_{h}} \int_{\WH{K}_{h}} \Big(r_h([u_h]) + \beta \cdot n_{h}^+l_h\big([u_h]\big) \Big) \cdot \Big( r_h([v_h]) + \beta \cdot n_{h}^+l_h\big([v_h]\big) \Big)\dAhk,
\end{align*}

and the surface LDG form can be written as
\begin{align}\label{eq:LDGGammahForm}
\ADG(&u_h,v_h)\notag \\
&=\sum_{\WH{K}_{h}\in\WH{\calT}_{h}} \int_{\WH{K}_{h}} \nabla_{\Gamma_{h}^{k}}u_h\cdot \nabla_{\Gamma_{h}^{k}}v_h + u_{h} v_{h}\dAhk\notag \\
&- \sum_{\WH{e}_{h}\in\WH{\calE}_{h}} \int_{\WH{e}_{h}} [u_h]\{\nabla_{\Gamma_{h}^{k}}v_h;n_{h}\}-\{\nabla_{\Gamma_{h}^{k}}u_h;n_{h}\}[v_h]\dshk \notag \\
&+\sum_{\WH{e}_{h}\in\WH{\calE}_{h}} \int_{\WH{e}_{h}} \bigg(-[\nabla_{\Gamma_{h}^{k}}u_h;n_{h}]\beta\cdot n_{h}^+[v_h] - \beta\cdot n_{h}^+[u_h][\nabla_{\Gamma_{h}^{k}}v_h;n_{h}] +\beta_{\WH{e}_{h}}[u_h][v_h] \bigg)\dshk \notag \\
& +\sum_{\WH{K}_{h}\in\WH{\calT}_{h}} \int_{\WH{K}_{h}} \Big(r_h([u_h]) + \beta\cdot n_{h}^+l_h\big([u_h]\big) \Big) \cdot \Big( r_h([v_h]) + \beta \cdot n_{h}^+l_h\big([v_h]\big) \Big)\dAhk.
\end{align}

\begin{remark}
In the flat case, for which we have $n_{h}^+=-n_{h}^-$, all of the surface DG methods yield the corresponding ones found in \cite{arnold2002unified}.
\end{remark}

\begin{remark}
Notice that for all of our choices of the numerical fluxes $\WH{u}$ and $\WH{\sigma}$, we have that $[\WH{u}] = 0$ and $[\WH{\sigma};n_{h}] = 0$.  In addition, they are consistent with the corresponding fluxes in the flat case given in \cite{arnold2002unified} with the exception of those of the surface LDG method. In the latter case, the equivalence does not hold because all the surface trace operators are scalars and they cannot be combined in the same way as the corresponding LDG fluxes in the flat case. 
\end{remark}

\section{Technical tools}
In this section we introduce the necessary tools and geometric relations needed to work on discrete domains and prove boundedness and stability of the bilinear forms, following the framework introduced in \cite{dziuk1988finite}. 
\subsection{Surface lifting}\label{sec:SurfaceLifting}
For any function $w$ defined on $\Gamma_{h}^{k}$ we define the surface \emph{lift} onto $\Gamma$ by
\[ w^{\ell}(\xi)  = w(x(\xi)),\ \xi \in \Gamma, \]
where, thanks to the invertibility of (\ref{eq:uniquePoint}), $x(\xi)$ is defined as the unique solution of
\[ x(\xi) = \pi(x) + d(x)\nu(\xi).\]
In particular, for every $\WH{K}_{h} \in \WH{\calT}_{h}$ there is a unique curved triangle $\WH{K}_{h}^{\ell}  = \pi(\WH{K}_{h}) \subset \Gamma$. We may then define the regular, conforming triangulation $\WH{\calT}_{h}^{\ell}$ of $\Gamma$ given by
\[ \Gamma = \bigcup_{\WH{K}_{h}^{\ell} \in \WH{\calT}_{h}^{\ell}} \WH{K}_{h}^{\ell}. \]
The triangulation $\WH{\calT}_{h}^{\ell}$ of $\Gamma$ is thus induced by the triangulation $\WH{\calT}_{h}$ of $\Gamma_{h}^{k}$ via the surface lift operator. Similarly, we denote by $\WH{e}_{h}^{\ell} =\pi(\WH{e}_{h}) \in \WH{\calE}_{h}^{\ell}$ the unique curved edge associated to $\WH{e}_{h}$.
The function space for surface lifted functions is chosen to be given by
$$\WH{S}_{hk}^{\ell}=\{\chi\in L^2(\Gamma) : \chi=\WH{\chi}^{\ell} \text{ for some } \WH{\chi}\in\WH{S}_{hk}\}.$$
We define the discrete right-hand side $f_{h}$ such that $f_{h}^{\ell} = f$. We also denote by $w^{-\ell} \in \WH{S}_{hk}$ the \emph{inverse} surface lift of some function $w \in \WH{S}_{hk}^{\ell}$ satisfying $(w^{-\ell})^{\ell} = w$.
\\
\\ 
One can show that for $v_{h}$ defined on $\Gamma_{h}^{k}$, we have that
\begin{equation*}
\nabla_{\Gamma_{h}^{k}} v_{h} = P_{h}(I - d H) P \nabla_{\Gamma} v_{h}^{\ell}.  
\end{equation*}
Furthermore, let $\delta_{h}$ be the local area deformation when transforming $\WH{K}_{h}$ to $\WH{K}_{h}^{\ell}$, i.e.,  
\[\delta_{h}\dAhk = \dA,\] 
and let $\delta_{\WH{e}_{h}}$ be the local edge deformation when transforming $\WH{e}_{h}$ to $\WH{e}_{h}^{\ell}$, i.e., \[\delta_{\WH{e}_{h}}\dshk= \ds.\] 
Finally, let
\begin{equation*}
R_{h} = \frac{1}{\delta_{h}} P (I - d H) P_{h} (I - d H) P. 
\end{equation*}
Then one can show that
\begin{align}\label{eq:liftingOfStiffnessTerm} 
\int_{\Gamma_{h}^{k}} \nabla_{\Gamma_{h}^{k}}u_{h} \cdot \nabla_{\Gamma_{h}^{k}}v_{h} + u_{h} v_{h} \dAh = \int_{\Gamma} R_{h} \nabla_{\Gamma}u_{h}^{\ell} \cdot \nabla_{\Gamma}v_{h}^{\ell} + \delta_{h}^{-1} u_{h}^{\ell} v_{h}^{\ell} \dA.  
\end{align}

\subsection{Geometric estimates}
We next prove some geometric error estimates relating $\Gamma$ to $\Gamma_{h}^{k}$.
\begin{lemma} \label{Gamma2GammahSmall}
Let $\Gamma$ be a compact smooth connected and oriented surface in $\mathbb{R}^{3}$ and let $\Gamma_{h}^{k}$ be its Lagrange interpolant of degree $k$.  Furthermore, we denote by $n^{+/-}$ the unit (surface) conormals to respectively $\WH{e}_{h}^{l+/-}$. Then, for sufficiently small $h$, we have that
\begin{subequations}
\begin{align}
\norm{d}_{L^{\infty}(\Gamma_{h}^{k})} &\lesssim h^{k+1},\label{1rel}\\
\norm{1-\delta_{h}}_{L^{\infty}(\Gamma_{h}^{k})} &\lesssim h^{k+1},\label{2rel}\\
\norm{\nu-\nu_{h}}_{L^{\infty}(\Gamma_{h}^{k})} &\lesssim h^{k},\label{3rel}\\
\norm{P-R_{h}}_{L^{\infty}(\Gamma_{h}^{k})} &\lesssim h^{k+1},\label{4rel}\\
\norm{1-\delta_{\WH{e}_{h}}}_{L^{\infty}(\WH{\mathcal{E}}_{h})} &\lesssim h^{k+1},\label{5rel}\\
\sup_{\WH{K} \in \WH{\mathcal{T}}_{h}}\norm{P-R_{\WH{e}_{h}}}_{L^{\infty}(\partial \WH{K}_{h})} &\lesssim h^{k+1},\label{6rel}\\
\norm{n^{+/-} -P n_{h}^{+/-}}_{L^{\infty}(\WH{\mathcal{E}}_{h})} &\lesssim  h^{k+1},\label{7rel}
\end{align}
\end{subequations}
where $R_{\WH{e}_{h}}  =  \frac{1}{\delta_{\WH{e}_{h}}}P(I-d H)P_{h}(I-d H)$.
\end{lemma}

For the sake of readability, we postpone the proof of Lemma \ref{Gamma2GammahSmall} to Appendix A. 

\subsection{Boundedness and stability}

We define the space of piecewise polynomial functions on $\Gamma_{h}$ as
$$\WT{S}_{hk}= \{\WT{\chi}\in L^2(\Gamma_{h}):\WT{\chi}|_{\WT{K}_{h}}\in \mathbb{P}^{k}(\WT{K}_{h})\ \ \ \forall \WT{K}_h\in\WT{\calT}_h\}.$$

We recall the following useful result from \cite{demlow2009higher}:

\begin{lemma}\label{Lemma4_2}
Let $v\in H^j(\WH{K}_h)$, $j\geq 2$, and let $\WT{v}=v \circ \pi_k$. Then, for $h$ small enough, we have that
\begin{subequations}
\begin{align}
\norm{v^{\ell}}_{L^2(\WH{K}_h^{\ell})}\sim&\norm{v}_{L^2(\WH{K}_h)}\sim\norm{\WT{v}}_{L^2(\WT{K}_h)},\label{normEquiv}\\
\norm{\nabla_{\Gamma}v^{\ell}}_{L^2(\WH{K}_{h}^{\ell})}\sim&\norm{\nabla_{\Gamma_h^k}v}_{L^2(\WH{K}_h)}
\sim\norm{\nabla_{\Gamma_h}\WT{v}}_{L^2(\WT{K}_h)},\label{semiNormEquiv}\\
\norm{D^j_{\Gamma_h^k}v}_{L^2(\WH{K}_h)}\lesssim& \sum_{1\leq m\leq j} \norm{D_\Gamma^m v^{\ell}}_{L^2(\WH{K}_h^{\ell})},\label{derEstimGammahk}\\
\norm{D^j_{\Gamma_h}\WT{v}}_{L^2(\WT{K}_h)}\lesssim& \sum_{1\leq m\leq j} \norm{D_{\Gamma_h^k}^m v}_{L^2(\WH{K}_h)}.\label{derEstimGammah}
\end{align}
\end{subequations}
\end{lemma}
We will also need the following inverse inequality, adapted from \cite[Thm 3.2.6]{Ciarlet2002}.
\begin{lemma}\label{thm:ciarletInverse}
Let $l, m$ be two integers such that $0 \leq l\leq m$. Then,
$$|v_h|_{H^m(\WT{K}_{h})} \lesssim h_{\WT{K}_{h}}^{l-m}|v_h|_{H^l(\WT{K}_h)} \qquad \forall v_h\in \WT{S}_{hk}.$$
\end{lemma}
Finally, we prove the following trace inequality:
\begin{lemma}\label{traceInequalityGammahk} 
For sufficiently small $h$, we have that
$$\norm{\nabla_{\Gamma_h^k}\WH{w}_h}^2_{L^2(\partial \WH{K}_{h})}\lesssim  h^{-1}\norm{\nabla_{\Gamma_h^k} \WH{w}_h}^2_{L^2(\WH{K}_{h})} \ \ \ \ \forall \WH{w}_h\in \WH{S}_{hk}.$$
\end{lemma}
\begin{proof}
Defining $\delta_{\WT{e}_{h}}  = \ds/\dsh$ and $\delta_{\WT{e}_{h} \rightarrow \WH{e}_{h}}  = \dshk/\dsh$, using (\ref{5rel}) and a Taylor expansion argument, we obtain
$$|1-\delta_{\WT{e}_{h} \rightarrow \WH{e}_{h}}|
=\abs{1-\frac{\delta_{\WT{e}_{h}}}{\delta_{\WH{e}_{h}}}}
=\abs{1-\frac{1+O(h^{2})}{1+O(h^{k+1})}}\lesssim h^2.$$
Now let $\WT{w}_h\in \WT{S}_{hk}$ be such that $\WT{w}_h=\WH{w}_h\circ \pi_k$.
From (2.21) in \cite{demlow2009higher} we have that
\begin{align}\label{eq:Gammah2GammahkEquivalence}
\nabla_{\Gamma_h^k}\WH{w}_h \lesssim  \nabla_{\Gamma_h}\WT{w}_h,
\end{align}
provided $h$ is sufficiently small.
Applying the trace theorem for polynomial functions on $\Gamma_{h}$ as given in Lemma 3.4 in \cite{dedner2013surfaces}, and the inverse inequality in Lemma \ref{thm:ciarletInverse} (with $l=1$ and $m=2$), we get
\begin{equation*}
\int_{\partial \WT{K}_h}|\nabla_{\Gamma_h}\WT{w}_h|^2\dsh \lesssim \frac{1}{h}\norm{\nabla_{\Gamma_h}\WT{w}_h}^2_{L^2(\WT{K}_h)}.
\end{equation*}
Surface lifting the left-hand side to $\Gamma_h^k$, making use of (\ref{eq:Gammah2GammahkEquivalence}) and using (\ref{semiNormEquiv}) for the right-hand side we have that
$$\int_{\partial \WH{K}_{h}}|\nabla_{\Gamma_h^k}\WH{w}_h|^2 \delta_{\WT{e}_{h} \rightarrow \WH{e}_{h}}^{-1} \dshk\lesssim \frac{1}{h}\norm{\nabla_{\Gamma_h^k}\WH{w}_h}^2_{L^2(\WH{K}_{h})}.$$
We thus obtain, using (\ref{5rel}), 
$$(1-Ch^2)\norm{\nabla_{\Gamma_h^k}\WH{w}_h}_{L^2(\partial \WH{K}_{h})}^2\lesssim \frac{1}{h}\norm{\nabla_{\Gamma_h^k}\WH{w}_h}^2_{L^2(\WH{K}_{h})},$$
which yields the desired result for $h$ small enough.
\end{proof}

In order to perform a unified analysis of the surface DG methods presented in Section~\ref{DGfluxes}, we introduce the stablization function
\begin{subnumcases}{S_{h}(u_{h},v_{h})=}  \sum_{\WH{e}_{h}\in\WH{\calE}_{h}}\beta_{\WH{e}_{h}} \int_{\WH{e}_{h}}[u_h][v_h] \dshk , & \label{S1} \\
\sum_{\WH{e}_{h} \in \WH{\calE}_{h}} \eta_{\WH{e}_{h}} \int_{\Gamma_{h}^{k}} r_{\WH{e}_{h}} ([u_h])\cdot r_{\WH{e}_{h}}([v_h])\dAhk , & \label{S2}
\end{subnumcases}
for $u_h,\ v_h\in \WH{S}_{hk}$, cf. also Table \ref{tab:DGMethods}. 

\begin{table}[!htb]
\begin{center}
\begin{tabular}{| c | c |}
  \hline
  Method & Stabilization function $S_{h}(\cdot,\cdot)$ \\
  \hline  
  \begin{tabular}{c} IP \cite{douglas1976interior} \\  NIPG \cite{riviere1999improved}\\  IIPG \cite{Dawson20042565} \\  LDG \cite{cockburn1998local} \end{tabular} & (\ref{S1})\\
  \hline
\begin{tabular}{c} Brezzi et al. \cite{brezzi1999discontinuous} \\ Bassi et al. \cite{bassi1997high2} \end{tabular} & (\ref{S2})\\
  \hline  
\end{tabular}
\end{center}
\caption{Stabilization function of the DG methods considered in our unified analysis.}
\label{tab:DGMethods}
\end{table}

The next result, together with the Lax-Milgram Lemma, guarantees that there exists a unique solution $u_{h} \in \WH{S}_{hk}$ of  \eqref{eq:PrimalForm} that satisfies the stability estimate
\begin{equation}\label{eq:stabilityEstimateGamma1}
\normDG{u_{h}} \lesssim  \norm{f_{h}}_{L^{2}(\Gamma_{h}^{k})},
\end{equation}
where the DG norm $\norm{\cdot}_{DG}$ is given by
\begin{equation}\label{def:DGNorm}
\begin{aligned}
\normDG{u_h}^2=\normh{u_h}^2+\normast{u_h}^2
&& \forall u_h\in \WH{S}_{hk},
\end{aligned}
\end{equation}
with
\begin{equation*}
\norm{u_h}_{1,h}^2 = \sum_{\WH{K}_{h}\in\WH{\calT}_h}\norm{u_h}^2_{H^1(\WH{K}_{h})},
\end{equation*}
and
\begin{equation*} 
\normast{u_h}^2 = S_{h}(u_{h},u_{h}),
\end{equation*}
where $S_{h}(\cdot,\cdot)$ depends on the method under investigation and is defined as in (\ref{S1})-(\ref{S2}).

We will now consider boundedness and stability of the bilinear forms $\ADG(\cdot,\cdot)$ corresponding to the surface DG methods given in Table \ref{tab:DGMethods}. We first state some estimates required for the analysis of the surface LDG method.

\begin{lemma}\label{re&leToAverage}
For any $v_h\in \WH{S}_{hk}$,
\begin{align*}\alpha \norm{r_{\WH{e}_{h}}([v_h])}_{L^2(\Gamma_h^k)}^2 &\lesssim \beta_{\WH{e}_{h}}\norm{[v_h]}_{L^2(\WH{e}_{h})}^2,\\
\alpha \norm{l_{\WH{e}_{h}}([v_h])}_{L^2(\Gamma_h^k)}^2 &\lesssim \beta_{\WH{e}_{h}}\norm{[v_h]}_{L^2(\WH{e}_{h})}^2,
\end{align*}
on each $\WH{e}_{h}\in \WH{\calE}_{h}$.
\end{lemma}
\begin{proof}
The proof is the same as \cite[Lemma 2.3]{antoniettiHouston} provided proper definition of the DG lift operators.
\end{proof}

\begin{lemma}\label{Lemma:boundStabGammakIP}
The bilinear forms $\ADG(\cdot, \cdot)$ corresponding to the surface DG methods given in Table \ref{tab:DGMethods} are continuous and coercive in the DG norm \eqref{def:DGNorm}, i.e.,
\begin{equation*}
\begin{aligned}
&\ADG(u_h,v_h)\lesssim  \normDG{u_h}\normDG{v_h},
&&\ADG(u_h,u_h)\gtrsim  \normDG{u_h}^2,
\end{aligned}
\end{equation*}
for every $u_h,v_h \in \WH{S}_{hk}$.

For the surface IP, Bassi et al. and IIPG methods, coercivity holds provided the penalty parameter $\alpha$ appearing in the definition of $\beta_{\WH{e}_{h}}$ or $\eta_{\WH{e}_{h}}$ in (\ref{penalityParameters}) is chosen sufficiently large.
\end{lemma}
\begin{proof}
For all the methods stabilized with $S_h(\cdot,\cdot)$ defined as in (\ref{S1}), Lemma \ref{traceInequalityGammahk} implies that
\begin{equation}\label{eq:EstimateMixTerm_a}
\begin{aligned}
\sum_{\WH{e}_{h}\in\WH{\calE}_{h}}\norm{[u_h]}_{L^2(\WH{e}_{h})}\norm{\{\nabla_{\Gamma_h^k} v_h; n_h\}}_{L^2(\WH{e}_{h})}\lesssim&
 \sum_{\WH{K}_{h}\in\WH{\calT}_h} \alpha^{-\half}  \normast{u_h}\norm{\nabla_{\Gamma_h^k} v_h}_{L^2(\WH{K}_{h})}\\
\lesssim & \alpha^{-\half} \normast{u_h}\normh{v_h},
\end{aligned}
\end{equation}
where the hidden constant depends on the degree of the polynomial approximation but not on the penalty parameters $\beta_{\WH{e}_{h}}$.
Otherwise, if $S_{h}(\cdot,\cdot)$ is given as in (\ref{S2}), we observe that for $u_h,v_h\in \WH{S}_{hk}$ we have that 
$$\sum_{\WH{e}_{h}\in\WH{\calE}_{h}}\int_{\WH{e}_h}[u_h] \{\nabla_{\Gamma_h^k} v_h; n_h\} \dshk =
\sum_{\WH{K}_{h}\in\WH{\calT}_h}\int_{\WH{K}_h}r_h([u_h])\cdot \nabla_{\Gamma_h^k} v_h \dAhk $$
and
\begin{equation}\label{relationrhre}
\norm{r_h(\phi)}_{L^2(\Gamma_{h}^{k})}^2=\norm{\sum_{\WH{e}_{h}\in\WH{\calE}_h}r_{\WH{e}_h}(\phi)}_{L^2({\Gamma_{h}^{k}})}^2\lesssim \sum_{\WH{e}_{h}\in\WH{\calE}_h}\norm{r_{\WH{e}_h}(\phi)}_{L^2({\Gamma_{h}^{k}})}^2.
\end{equation}
Hence, applying the Cauchy-Schwarz inequality, we obtain
\begin{equation}\label{eq:EstimateMixTerm_b}
\begin{aligned}
\sum_{\WH{K}_{h}\in\WH{\calT}_h}\norm{r_h([u_h])}_{L^2(\WH{K}_{h})}\norm{\nabla_{\Gamma_h^k} v_h}_{L^2(\WH{K}_{h})}\lesssim&
 \sum_{\WH{K}_{h}\in\WH{\calT}_h} \alpha^{-\half} \normast{u_h}\norm{\nabla_{\Gamma_h^k} v_h}_{L^2(\WH{K}_{h})}\\
\lesssim & \alpha^{-\half} \normast{u_h} \normh{v_h},
\end{aligned}
\end{equation}
where the hidden constant depends on the degree of the polynomial approximation but not on the penalty parameters $\eta_{\WH{e}_{h}}$. 
For the surface LDG method, using Lemma \ref{re&leToAverage}, Lemma \ref{traceInequalityGammahk} and the $L^\infty(\Gamma_h^k)$ bound on $\beta$, we obtain
$$\left|\int_{\WH{e}_{h}} [\nabla_{\Gamma_{h}^{k}}u_h;n_{h}]\beta\cdot n_{h}^+[v_h]\dshk\right| \lesssim \alpha^{-\half}\|\beta\|_{L^\infty(\Gamma_h^k)} \norm{\nabla_{\Gamma_{h}^{k}}u_h}_{L^2(\WH{K}_{h})} \normast{v_h},$$
$$\left|\int_{\WH{K}_{h}} r_h([u_h]) \cdot l_h(\beta\cdot n_{h}^+[u_h])\dshk\right| \lesssim \alpha^{-1}\|\beta\|_{L^\infty(\Gamma_h^k)} \normast{u_h} \normast{v_h},$$
and, in a similar way, the remaining quantities.
Continuity then follows from the Cauchy-Schwarz inequality and the above estimates.\\

We next show coercivity of the DG bilinear forms. For the surface NIPG method, stability follows straightforwardly from the Cauchy-Schwarz inequality. For the surface LDG method, we have that
\begin{align*}
\ADG(u_h,u_h)\geq& \normh{u_h}^2-2\sum_{\WH{e}_{h}\in\WH{\calE}_h^k}\int_{\WH{e}_{h}}\left|[u_h]\{\nabla_{\Gamma_h^k} u_h;n_h\}\right|\dshk\\ &-2\norm{\beta}_{L^\infty(\Gamma_h^k)}\sum_{\WH{e}_{h}\in\WH{\calE}_h^k}\int_{\WH{e}_{h}}\left|[u_h][\nabla_{\Gamma_h^k} u_h;n_h]\right|\dshk+\normast{u_h}^2.
\end{align*} 
For the other methods involving $S_{h}(\cdot, \cdot)$ defined as in (\ref{S1}) we obtain
\begin{align*}
\ADG(u_h,u_h)\geq& \normh{u_h}^2-2\sum_{\WH{e}_{h}\in\WH{\calE}_h^k}\int_{\WH{e}_{h}}\left|[u_h]\{\nabla_{\Gamma_h^k} u_h;n_h\}\right|\dshk+\normast{u_h}^2,
\end{align*}
otherwise, if $S_{h}(\cdot, \cdot)$ is given as in (\ref{S2}), we have that
\begin{align*}
\ADG(u_h,u_h)\geq& \normh{u_h}^2-2\sum_{\WH{K}_{h}\in\WH{\calT}_h^k}\int_{\WH{K}_{h}}\left|r_h([u_h]) \cdot \nabla_{\Gamma_h^k} u_h\right|\dAhk+\normast{u_h}^2.
\end{align*}
The result follows by making use of the corresponding boundedness estimates, using using Cauchy-Schwarz inequality and Young's inequalities and choosing the penalty parameter $\alpha$ sufficiently large.
\end{proof}

We now define the DG norm for functions in $\WH{S}_{hk}^{\ell}$ as follows:
\begin{equation}\label{def:DGNormLift}
\begin{aligned}
\normDG{u_h^{\ell}}^2=\normh{u_h^{\ell}}^2+\normast{u_h^{\ell}}^2
&& \forall u_h^{\ell}\in \WH{S}_{hk}^{\ell},
\end{aligned}
\end{equation}
with
\begin{equation*}
\norm{u_h^{\ell}}_{1,h}^2 = \sum_{\WH{K}_{h}^{\ell}\in\WH{\calT}_h^{\ell}}\norm{u_h^{\ell}}^2_{H^1(\WH{K}_{h}^{\ell})}, 
\end{equation*}
and
\begin{equation*}
\normast{u_h^{\ell}}^2 = S_{h}^{\ell}(u_{h}^{\ell},u_{h}^{\ell}),
\end{equation*}
where $S_{h}^{\ell}(\cdot, \cdot)$ is defined according in (\ref{S1})-(\ref{S2}) but on $\Gamma$, i.e.,
\begin{subnumcases}{S_{h}^{\ell}(u_{h}^{\ell},v_{h}^{\ell})=}
\sum_{\WH{e}_{h}\in\WH{\calE}_h}\beta_{\WH{e}_{h}} \int_{\WH{e}_{h}^{\ell}}\delta_{\WH{e}_{h}}^{-1}[u_h^{\ell}][v_h^{\ell}] \ds , & \label{S3}\\
\sum_{\WH{e}_{h}\in\WH{\calE}_h}\eta_{\WH{e}_{h}} \int_{\Gamma} \delta_{h}^{-1} \big(r_{\WH{e}_{h}}([u_h])\big)^{\ell} \cdot \big(r_{\WH{e}_{h}}([v_h])\big)^{\ell}\dA , & \label{S4}
\end{subnumcases}
for $u_{h}^{\ell},\ v_{h}^{\ell}\in \WH{S}_{hk}^{\ell}$.
\begin{lemma}
Let $u_{h} \in \WH{S}_{hk}$ satisfy (\ref{eq:stabilityEstimateGamma1}). Then $u_{h}^{\ell} \in \WH{S}_{hk}^{\ell}$ satisfies
\begin{equation}\label{eq:stabilityEstimateGamma2}
\norm{u_{h}^{\ell}}_{DG} \lesssim  \norm{f}_{L^{2}(\Gamma)},
\end{equation}
for $h$ small enough.
\end{lemma}
\begin{proof}
We first show that for any function $v_h\in \WH{S}_{hk}$, for sufficiently small $h$,
\begin{align}\label{eq:normestimate}
\norm{v_h^{\ell}}_{DG} \lesssim \normDG{v_h}.
\end{align}
The $\normh{\cdot}^2$ component of the DG norm is dealt with in exactly the same way as in \cite{demlow2009higher}. For the $\normast{\cdot}^2$ component of the DG norm we have that
 

\begin{align*} 
\int_{\WH{e}_{h}} [v_h]^{2} \dshk = \int_{\WH{e}_{h}^{\ell}} \delta_{\WH{e}_{h}}^{-1} [v_h^{\ell}]^{2} \ds\quad\text{and}\quad \int_{\Gamma_h^k} |r_{h}([v_h])|^2 \dAhk = \int_{\Gamma} \delta_{h}^{-1} |r_{h}([v_h])^{\ell}|^2 \dA,
\end{align*}
which straightforwardly yields (\ref{eq:normestimate}). Making use of the discrete stability estimate (\ref{eq:stabilityEstimateGamma1}) and noting that, by Lemma \ref{re&leToAverage}, $\norm{f_{h}}_{L^2(\Gamma_h^k)} \lesssim \norm{f_{h}^{\ell}}_{L^2(\Gamma)} = \norm{f}_{L^2(\Gamma)}$, we get the desired result. 
\end{proof}


For each of the surface DG bilinear forms given in Table \ref{tab:DGMethods}, we define a corresponding bilinear form on $\Gamma$ induced by the surface lifted triangulation $\WH{\cal{T}}_{h}^{\ell}$ which is well defined for functions $w,v \in H^2(\Gamma)+\WH{S}_{hk}^{\ell}$. For the surface IP bilinear form (\ref{eq:InteriorPenaltyGammahForm}), we define 
\begin{align}\label{eq:InteriorPenaltyGammaForm}
\AC(w,v) & =  \sum_{\WH{K}_{h}^{\ell} \in \WH{\calT}_{h}^{\ell}}\int_{\WH{K}_{h}^{\ell}}\nabla_{\Gamma}w\cdot \nabla_{\Gamma}v+ w v \dA - \sum_{\WH{e}_{h}^{\ell} \in \WH{\calE}_{h}^{\ell}}\int_{\WH{e}_{h}^{\ell}}[w]\{\nabla_{\Gamma}v; n \} + [v]\{\nabla_{\Gamma}w; n \} \ds \notag \\
&+ \sum_{\WH{e}_{h}^{\ell} \in \WH{\calE}_{h}^{\ell}}\int_{\WH{e}_{h}^{\ell}}\delta_{\WH{e}_{h}}^{-1}\beta_{\WH{e}_{h}}[w][v] \ds,
\end{align}
where $n^{+}$ and $n^{-}$ are respectively the unit surface conormals to $\WH{K}_{h}^{\ell+}$ and $\WH{K}_{h}^{\ell-}$ on $\WH{e}_{h}^{\ell} \in \WH{\calE}_{h}^{\ell}$. For the Brezzi et al. bilinear form (\ref{eq:BrezziEtAlGammahForm}), we define
\begin{align}\label{eq:BrezziEtAlGammaForm}
\AC(w,v) & =  \sum_{\WH{K}_{h}^{\ell} \in \WH{\calT}_{h}^{\ell}}\int_{\WH{K}_{h}^{\ell}}\nabla_{\Gamma}w\cdot \nabla_{\Gamma}v+ w v \dA\notag \\ 
&+ \sum_{\WH{K}_{h}^{\ell} \in \WH{\calT}_{h}^{\ell}}\int_{\WH{K}_{h}^{\ell}} \delta_{h}^{-1}\eta_{\WH{e}_{h}} r_{\WH{e}_{h}}([w^{-\ell}])^{\ell}\cdot r_{\WH{e}_{h}}([v^{-\ell}])^{\ell} + \delta_{h}^{-1}\big(r_h([w^{-\ell}])\big)^{\ell}\cdot \big(r_h([v^{-\ell}])\big)^{\ell} \dA\notag \\
&- \sum_{\WH{e}_{h}^{\ell} \in \WH{\calE}_{h}^{\ell}}\int_{\WH{e}_{h}^{\ell}}[w]\{\nabla_{\Gamma}v; n \} + [v]\{\nabla_{\Gamma}w; n \}\ - \delta_{\WH{e}_{h}}^{-1}\beta_{\WH{e}_{h}}[w][v] \ds.
\end{align}
For the surface LDG bilinear form (\ref{eq:LDGGammahForm}), we define
\begin{align}\label{eq:LDGGammaForm}
&\AC(w,v)=\sum_{\WH{K}_{h}^{\ell}\in\WH{\calT}_{h}^{\ell}} \int_{\WH{K}_{h}^{\ell}} \nabla_{\Gamma}w \cdot \nabla_{\Gamma}v + w v \dA - \sum_{\WH{e}_{h}^{\ell}\in\WH{\calE}_{h}^{\ell}} \int_{\WH{e}_{h}^{\ell}} [w]\{\nabla_{\Gamma}v;n\}-\{\nabla_{\Gamma}w;n\}[v] \ds \notag \\
&+\sum_{\WH{e}_{h}^{\ell}\in\WH{\calE}_{h}^{\ell}} \int_{\WH{e}_{h}^{\ell}} \bigg(-\delta_{\WH{e}_{h}}^{-1}[\nabla_{\Gamma}w;n]\beta\cdot n_{h}^{\ell +}[v] - \delta_{\WH{e}_{h}}^{-1}\beta\cdot n_{h}^{\ell +}[w][\nabla_{\Gamma}v;n] +\delta_{\WH{e}_{h}}^{-1}\beta_{\WH{e}_{h}}[w][v] \bigg) \ds \notag \\
& +\sum_{\WH{K}_{h}^{\ell}\in\WH{\calT}_{h}^{\ell}} \int_{\WH{K}_{h}^{\ell}} \Big(r_h([w^{-\ell}]) + \beta\cdot n_{h}^{\ell+} l_h\big([w^{-\ell}]\big) \Big)^{\ell} \cdot \Big( r_h([v^{-\ell}]) + \beta \cdot n_{h}^{\ell+} l_h\big([v^{-\ell}]\big) \Big)^{\ell} \dA.
\end{align}

The corresponding bilinear forms for the other surface DG methods can be derived in a similar manner. Since we assume that the weak solution $u$ of \eqref{eq:weakH1} belongs to $H^{2}(\Gamma)$ they all satisfy
\begin{equation} \label{eq:InteriorPenaltyGamma}
\AC(u,v) = \sum_{\WH{K}_{h}^{\ell} \in \WH{\calT_{h}}^{\ell}}\int_{\WH{K}_{h}^{\ell}}f v \dA,\ \ \ \ \ \forall v \in H^{2}(\Gamma) + \WH{S}^{\ell}_{hk}.
\end{equation} 

Finally, we require the following stability estimate for $\AC(\cdot, \cdot)$, which follows by applying similar arguments as those found in the proof of Lemma \ref{Lemma:boundStabGammakIP}.
\begin{lemma}\label{Lemma:boundStabGammaIP}
The bilinear forms $\AC(\cdot, \cdot)$ induced by the surface DG methods given in Table \ref{tab:DGMethods} are coercive in the DG norm \eqref{def:DGNormLift}, i.e.,
\begin{equation}\label{eq:stabilityGammaIP}
\norm{w_{h}^{\ell}}_{DG}^{2} \lesssim \AC(w_{h}^{\ell},w_{h}^{\ell})
\end{equation}
for all $w_{h}^{\ell} \in \WH{S}_{hk}^{\ell}$ if, for the surface IP, Bassi et al. and IIPG methods, the penalty parameter $\alpha$ appearing in the definition of $\beta_{\WH{e}_{h}}$ or $\eta_{\WH{e}_{h}}$ in (\ref{penalityParameters}) is chosen sufficiently large.
\end{lemma}

\section{Convergence}
\label{sec:Convergence}
We now state the main result of this paper.
\begin{theorem}\label{aprioriErrorEstimateIP}
Let $u \in H^{k+1}(\Gamma)$ and $u_{h} \in \WH{S}_{hk}$ denote the solutions to (\ref{eq:weakH1}) and (\ref{eq:PrimalFormulation}), respectively. Let $\eta=0$ for IIPG, NIPG formulations and let $\eta=1$ otherwise. Then,
\[ \norm{u-u_{h}^{\ell}}_{L^{2}(\Gamma)} + h^{\eta}\norm{u-u_{h}^{\ell}}_{DG} \lesssim h^{k+\eta}(\norm{f}_{L^{2}(\Gamma)}+\norm{u}_{H^{k+1}(\Gamma)}),\]
provided the mesh size $h$ is small enough and the penalty parameter $\alpha$ is large enough for the surface IP, Bassi et al. and IIPG methods.
\end{theorem}

The proof will follow an argument similar to the one outlined in \cite{arnold2002unified}. Using the stability result (\ref{eq:stabilityGammaIP}), we have that
\begin{equation} \label{eq:first_sec4}
\norm{\phi_{h}^{\ell}- u_{h}^{\ell}}_{DG}^{2} \lesssim
  \AC(\phi_{h}^{\ell}-u_{h}^{\ell},\phi_{h}^{\ell}- u_{h}^{\ell}) = \AC(u-u_{h}^{\ell},\phi_{h}^{\ell}-u_{h}^{\ell}) +  \AC(\phi_{h}^{\ell}-u,\phi_{h}^{\ell}-u_{h}^{\ell}),
\end{equation}
where $\phi_{h}^{\ell} \in \WH{S}^{\ell}_{hk}$. Since we do not directly have Galerkin orthogonality the first term on the right-hand side of (\ref{eq:first_sec4}) is not zero and its estimation will be the main part of this section.
%
The second term is dealt with in the following way:
following \cite{demlow2009higher}, for $\widehat{w}\in H^2(\Gamma_{h}^{k})$, we define the interpolant $\WH{I}_h^k : C^0(\Gamma_h^k)\rightarrow \WH{S}_{hk}$ by $$\WH{I}_{h}^{k}\widehat{w}=\WT{I}_{h}^{k}(\WH{w}\circ\pi_k),$$
where $\WT{I}_h^k : C^0(\Gamma_h)\rightarrow \WT{S}_{hk}$ is the standard Lagrange interpolant of degree $k$. 
We also define the interpolant $I_h^k : C^0(\Gamma)\rightarrow \WH{S}_{hk}^{\ell}$ by 
$$I_h^k w  = \WH{I}_h^k (w \circ \pi ).$$

\begin{lemma}\label{interpolationEstimate}
Let $w \in H^{m}(\Gamma)$ with $2\leq m \leq k+1$. Then for $i=0, 1,$
\[|w - I_h^k w|_{H^i(\WH{K}_{h}^{\ell})} \lesssim h^{m-i}\norm{w}_{H^m(\WH{K}_{h}^{\ell})}. \]
\end{lemma}
\begin{proof}
The proof follows easily by combining standard estimates for the Lagrange interpolant on $\Gamma_h$ with Lemma \ref{Lemma4_2}. See \cite{demlow2009higher} for further details.
\end{proof}

\begin{lemma}\label{interpolationEstimateNew}
Let $w\in H^{m}(\Gamma)$ with $2\leq m \leq k+1$. Then, for sufficiently small $h$, we have that
$$\norm{w-I_h^k w}_{L^2(\partial \WH{K}_h^{\ell})}^2+h^2\norm{\nabla_{\Gamma}(w-I_h^k w)}_{L^2(\partial \WH{K}_h^{\ell})}^2\lesssim h^{2m-1}\norm{w}^2_{H^{m}(\WH{K}_h^{\ell})}.$$
\end{lemma}

\begin{proof} Fix an arbitrary element $\WH{K}_{h}^{\ell} \in \WH{\mathcal{T}}_{h}^{\ell}$. We then define $\WH{w}\in H^{m}(\WH{K}_{h})$ and ${\WT{w}\in H^{m}(\WT{K}_h)}$ such that $w=\WH{w}\circ \pi$ and $\WT{w}=\WH{w}\circ \pi_k$.

Applying the trace theorem on $\WT{K}_h\in\WT{\calT}_h$ we get
$$\int_{\partial \WT{K}_h}|\nabla_{\Gamma_h}(\WT{w}-\WT{I}_h^k \WT{w})|^2\dsh\lesssim \bigg(\frac{1}{h}\int_{\WT{K}_h}|\nabla_{\Gamma_h}(\WT{w}-\WT{I}_h^k \WT{w})|^2\dAh +h\int_{\WT{K}_h}|\nabla^2_{\Gamma_h}(\WT{w}-\WT{I}_h^k \WT{w})|^2\dAh \bigg).$$
Applying a classical interpolation result for the right-hand side, we obtain
$$\int_{\partial \WT{K}_h}|\nabla_{\Gamma_h}(\WT{w}-\WT{I}_h^k \WT{w})|^2\dsh\lesssim h^{2m-3}|\WT{w}|^2_{H^{m}(\WT{K}_h)}.$$
Then, lifting the left-hand side on $\Gamma_h^k$ as in Lemma \ref{traceInequalityGammahk} and using (\ref{semiNormEquiv}) with (\ref{derEstimGammah}) we have
$$(1-Ch^2)\int_{\partial \WH{K}_{h}}|\nabla_{\Gamma_{h}^{k}}(\WH{w}-\WH{I}_h^k \WH{w})|^2 \dshk\lesssim h^{2m-3}\norm{\WH{w}}^2_{H^{m}(\WH{K}_{h})}.$$
In the same way, we lift the left-hand side onto $\Gamma$ and use (\ref{semiNormEquiv}) with (\ref{derEstimGammah}):
$$(1-Ch^{k+1})(1-Ch^2)\norm{\nabla_{\Gamma}(w-I_h^k w)}_{L^2(\partial \WH{K}_h^{\ell})}^2\lesssim h^{2m-3}\norm{w}^2_{H^{m}(\WH{K}_h^{\ell})}.$$
Then, proceeding similarly with $\norm{w-I_h^k w}_{L^2(\partial \WH{K}_h^{\ell})}^2$, we get the desired result for $h$ small enough.
\end{proof}


These interpolation estimates allow us to derive the following estimates for $\AC(\cdot, \cdot)$: 

\begin{lemma}\label{ApproxErrorIterpLemma}
Let $u \in H^{m}(\Gamma)$ and $w \in H^{n}(\Gamma)$ with $2\leq m, n \leq k+1$. Then, for all $v_h^{\ell}\in \WH{S}_{hk}^{\ell}$, we have that
\begin{align}
\AC(u-I_h^k u, v_h^{\ell})&\lesssim  h^{m-1} \norm{u}_{H^{m}(\Gamma)} \norm{v_{h}^{\ell}}_{DG}, \label{eq:approxErrorInterp} \\
\AC(u-I_h^k u, w-I_h^k w)&\lesssim  h^{m+n-2} \norm{u}_{H^{m}(\Gamma)} \norm{w}_{H^{n}(\Gamma)}. \label{eq:approxErrorInterpBis}
\end{align}
\end{lemma}
\begin{proof}
Since $u \in H^{m}(\Gamma) \subset C^0(\Gamma)$ for $m \geq 2$ and $I_h^k u \in C^0(\Gamma)$, we have that $[\WH{u}-\WH{I}_h^k \WH{u}]=0$ on each $\WH{e}_{h} \in \WH{\calE}_{h}$, where $\WH{u}^{\ell}=u$.
Than, using Cauchy-Schwarz formula in the definition of $r_{\WH{e}_{h}}$ and $l_{\WH{e}_{h}}$, we get
\begin{align*}
\norm{r_{\WH{e}_{h}}([\WH{u}-\WH{I}_h^k \WH{u}])}^2_{L^2(\Gamma_{h}^{k})} &\leq \norm{[\WH{u}-\WH{I}_h^k \WH{u}]}_{L^2(\WH{e}_{h})}\norm{\{r_{\WH{e}_{h}}([\WH{u}-\WH{I}_h^k \WH{u}]);n_{h}\}}_{L^2(\WH{e}_{h})}=0;\\
\norm{l_{\WH{e}_{h}}([\WH{u}-\WH{I}_h^k \WH{u}])}^2_{L^2(\Gamma_{h}^{k})} &\leq \norm{[\WH{u}-\WH{I}_h^k \WH{u}]}_{L^2(\WH{e}_{h})} \norm{[l_{\WH{e}_{h}}([\WH{u}-\WH{I}_h^k \WH{u}]);n_{h}]}_{L^2(\WH{e}_{h})} = 0.
\end{align*}
Then, following the proof of Lemma \ref{Lemma:boundStabGammakIP}, it is easy to obtain (\ref{eq:approxErrorInterp}) and (\ref{eq:approxErrorInterpBis}) from Lemma \ref{interpolationEstimate} and Lemma \ref{interpolationEstimateNew}.
\end{proof}

For the first term on the right-hand side of (\ref{eq:first_sec4}), we require the following \emph{perturbed} Galerkin orthogonality result:
\begin{lemma} \label{PerturbedGalerkinOrthogonality}
Let $u \in H^s(\Gamma)$, $s \geq 2$, and $u_{h} \in \WH{S}_{hk}$ denote the solutions to (\ref{eq:weakH1}) and (\ref{eq:PrimalFormulation}), respectively. We define
the functional $E_h$ on $\WH{S}^{\ell}_{hk}$ by
\begin{equation*} 
E_{h}(v_h^{\ell})  = \AC(u-u^{\ell}_{h},v^{\ell}_{h}).
\end{equation*}

Then, for all surface DG methods apart from LDG, $E_{h}$ can be written as
\begin{align}\label{eq:errorFunctional1}
E_{h}(v_{h}^{\ell}) =& \sum_{\WH{K}_{h}^{\ell} \in \WH{\calT}^{\ell}_{h}}\int_{\WH{K}_{h}^{\ell}}(R_{h}-P)\nabla_{\Gamma}u_{h}^{\ell} \cdot\nabla_{\Gamma}v_{h}^{\ell}+\left(\delta_{h}^{-1}-1\right)u_{h}^{\ell}v_{h}^{\ell} + \left(1-\delta_{h}^{-1}\right)fv_{h}^{\ell} \dA \notag \\
&+ \sum_{\WH{e}_{h}^{\ell} \in \WH{\calE}_{h}^{\ell}}\int_{\WH{e}_{h}^{\ell}} [u_{h}^{\ell}]\left(\{ \nabla_{\Gamma}v_{h}^{\ell}; n \} - \{ \delta_{\WH{e}_{h}}^{-1} P_{h}(I-dH)P\nabla_{\Gamma}v_{h}^{\ell}; n_{h}^{\ell} \} \right) \ds \notag \\ 
&+ \sum_{\WH{e}_{h}^{\ell} \in \WH{\calE}^{\ell}_{h}}\int_{\WH{e}_{h}^{\ell}} [v_{h}^{\ell}]\left(\{ \nabla_{\Gamma}u_{h}^{\ell}; n \} - \{ \delta_{\WH{e}_{h}}^{-1} P_{h}(I-dH)P\nabla_{\Gamma}u_{h}^{\ell}; n_{h}^{\ell} \} \right) \ds
\end{align}
where $R_{h}$ is given as in Lemma \ref{Gamma2GammahSmall}. The functional corresponding to the surface LDG method can be written as
\begin{align}\label{eq:errorFunctional2}
E_{h}(v_{h}^{\ell}) =& ( \ref{eq:errorFunctional1} \notag ) + \sum_{\WH{e}_{h}^{\ell} \in \WH{\calE}^{\ell}_{h}}\int_{\WH{e}_{h}^{\ell}}\delta_{\WH{e}_{h}}^{-1}\beta\cdot n_{h}^{\ell +}[v_{h}^{\ell}]\left( [ \nabla_{\Gamma}u_{h}^{\ell}; n ] - [ P_{h}(I-dH)P\nabla_{\Gamma}u_{h}^{\ell}; n_{h}^{\ell} ] \right) \ds \notag \\
&+ \sum_{\WH{e}_{h}^{\ell} \in \WH{\calE}^{\ell}_{h}}\int_{\WH{e}_{h}^{\ell}}\delta_{\WH{e}_{h}}^{-1}\beta\cdot n_{h}^{\ell +}[u_{h}^{\ell}]\left( [ \nabla_{\Gamma}v_{h}^{\ell}; n ] - [ P_{h}(I-dH)P\nabla_{\Gamma}v_{h}^{\ell}; n_{h}^{\ell} ] \right) \ds.
\end{align}

Furthermore,
\begin{equation} \label{eq:Eh_quad}
|E_{h}(v_{h}^{\ell})| \lesssim h^{k+1}\norm{f}_{L^{2}(\Gamma)}\norm{v_{h}^{\ell}}_{DG}.
\end{equation}
\end{lemma}
\begin{proof}
The proof is similar to that of Lemma 4.2 in \cite{dedner2013surfaces} which considered a piecewise linear approximation of the surface. The expression for the error functional $E_{h}$ is obtained by first noting that the solution $u$ of (\ref{eq:weakH1}) satisfies (\ref{eq:InteriorPenaltyGamma}) and then considering the difference between (\ref{eq:InteriorPenaltyGamma}) and (\ref{eq:PrimalFormulation}). This is done by first surface lifting the terms of (\ref{eq:PrimalForm}) onto $\Gamma$ in a similar fashion to (\ref{eq:liftingOfStiffnessTerm}). The estimate (\ref{eq:Eh_quad}) is then obtained by making use of the geometric estimates in Lemma \ref{Gamma2GammahSmall}. 
\end{proof}

\begin{remark}
Note that the error functional $E_{h}$ in Lemma \ref{PerturbedGalerkinOrthogonality} includes all of the terms present in the high order surface FEM setting (see \cite{demlow2009higher}) as well as additional terms arising from the surface DG methods. 
\end{remark}

\begin{proof}[Proof of Theorem \ref{aprioriErrorEstimateIP}]
Choosing the continuous interpolant $\phi_{h}^{\ell} = I^k_hu$, using the interpolation estimate (\ref{eq:approxErrorInterp})  
and the error functional estimate (\ref{eq:Eh_quad}), (\ref{eq:first_sec4}) can be bounded by
\begin{align*}
\norm{I^k_hu- u_{h}^{\ell}}_{DG}^{2}  &\lesssim 
    E_{h}(I^k_hu-u_{h}^{\ell}) + \AC(I^k_hu-u,I^k_hu-u_{h}^{\ell})\\  
&\lesssim h^{k+1}\norm{f}_{L^{2}(\Gamma)}\norm{I^k_hu-u_{h}^{\ell}}_{DG} +  h^{k} \norm{u}_{H^{k+1}(\Gamma)} \norm{I^k_hu-u_{h}^{\ell}}_{DG} \notag, 
\end{align*} 
which implies
\[ \norm{I^k_hu- u_{h}^{\ell}}_{DG} \lesssim h^k(\norm{f}_{L^{2}(\Gamma)} + \norm{u}_{H^{k+1}(\Gamma)}). \]
Recalling that $u-I^k_hu\in C^0(\Gamma)$, using Lemma~\ref{interpolationEstimate} we obtain
\begin{align*}
  \norm{u-u_h^{\ell}}_{DG} &\leq \norm{u-I^k_hu}_{DG} + \norm{I^k_hu- u_{h}^{\ell}}_{DG} 
 \lesssim h^k(\norm{f}_{L^{2}(\Gamma)}+ \norm{u}_{H^{k+1}(\Gamma)}).
\end{align*}
This concludes the first part of the proof.
In the case of $\eta=1$, to derive the $L^{2}$ estimate we first observe that the solution $z \in H^{2}(\Gamma)$ to the dual problem
\begin{equation}\label{eq:DualEllipticGamma}
-\Delta_{\Gamma} z + z = u-u_h^{\ell}
\end{equation}
satisfies
\begin{equation}\label{eq:DualAssumption}
\norm{z}_{H^{2}(\Gamma)} \lesssim  \norm{u-u_h^{\ell}}_{L^{2}(\Gamma)}.
\end{equation}
Then, using the symmetry of the bilinear function $\AC(\cdot,\cdot)$, we have that
\begin{align}\label{l2errorSquared}
\norm{u - u_{h}^{\ell}}_{L^{2}(\Gamma)}^2&=(u - u_{h}^{\ell},u - u_{h}^{\ell})_{\Gamma}=\AC(z,u - u_{h}^{\ell})\notag \\
&=\AC(u - u_{h}^{\ell},z) =\AC(u - u_{h}^{\ell},z-I_h^kz)+E_h(I_h^kz).
\end{align}
Using (\ref{eq:Eh_quad}), a triangle inequality and the interpolation estimate in Lemma \ref{interpolationEstimate}, we obtain
$$|E_h(I_h^kz)|\lesssim  h^{k+1} \norm{f}_{L^2(\Gamma)} \norm{I_h^k z}_{H^{1}(\Gamma)}\lesssim  h^{k+1} \norm{f}_{L^2(\Gamma)} \norm{z}_{H^{2}(\Gamma)}.$$
Hence, using (\ref{eq:DualAssumption}),
\begin{equation*}
|E_h(I_h^kz)|\lesssim  h^{k+1} \norm{f}_{L^2(\Gamma)} \norm{u - u_{h}^{\ell}}_{L^{2}(\Gamma)}
\end{equation*}
Making use of the continuity of $I_h^k z-z$ and $I_h^k u-u$, with the symmetry of the bilinear function $\AC(\cdot,\cdot)$, Lemma \ref{ApproxErrorIterpLemma} and the stability estimate (\ref{eq:DualAssumption}) we get
\begin{align*}
\AC(u - u_{h}^{\ell},z-I_h^kz&)=\AC(z-I_h^kz,u - u_{h}^{\ell}) \\
&\lesssim \AC(z-I_h^kz,I_h^k u - u_{h}^{\ell})+\AC(z-I_h^kz,u - I_h^k u)\\
&\lesssim h\norm{z}_{H^2(\Gamma)}\norm{I_h^k u - u_{h}^{\ell}}_{DG} + h^{k+1}\norm{z}_{H^2(\Gamma)} \norm{u}_{H^{k+1}(\Gamma)}\\
&\lesssim h \norm{z}_{H^2(\Gamma)}(\norm{I_h^k u - u}_{DG} + \norm{u - u_{h}^{\ell}}_{DG})  + h^{k+1}\norm{z}_{H^2(\Gamma)} \norm{u}_{H^{k+1}(\Gamma)}\\
&\lesssim (h^{k+1}\norm{u}_{H^{k+1}(\Gamma)} + h\norm{u - u_{h}^{\ell}}_{DG}) \norm{u - u_{h}^{\ell}}_{L^{2}(\Gamma)}.
\end{align*}
Combining the last two inequalities with (\ref{l2errorSquared}) yields
\begin{align*}
\norm{u - u_{h}^{\ell}}_{L^{2}(\Gamma)}^2&\lesssim \left( h\norm{u - u_{h}^{\ell}}_{DG} + h^{k+1} (\norm{f}_{L^2(\Gamma)}+\norm{u}_{H^{k+1}(\Gamma)})\right) \norm{u - u_{h}^{\ell}}_{L^{2}(\Gamma)},
\end{align*}
which gives us the desired $L^{2}$ estimate and concludes the proof.
In the case of $\eta=0$, we can obtain a sub-optimal bound using a similar procedure to that of \cite{arnold2002unified}.
\end{proof}

\section*{Acknowledgements}
This research has been supported by the British Engineering and Physical Sciences Research Council (EPSRC), Grant EP/H023364/1.
\FloatBarrier

\appendix
\section*{}

This section is devoted to prove Lemma \ref{Gamma2GammahSmall}.
\begin{proof}[Proof of Lemma \ref{Gamma2GammahSmall}]
Proofs of (\ref{1rel})-(\ref{4rel}) can be found in \cite[Prop. 2.3 and Prop. 4.1]{demlow2009higher}. The proof of (\ref{6rel}) will follow exactly the same lines as (\ref{4rel}) once we have proven (\ref{5rel}). Let $e$, $K$ be the reference segment [0,1] and the (flat) reference element, respectively, and let
$\WT{K}_h$, $\WH{K}_h$ and $\WH{K}_h^{\ell}$  be  elements in $\Gamma_h$, $\Gamma_{h}^{k}$ and $\Gamma$, respectively, such that $\pi_k(\WT{K}_h)=\WH{K}_h$ and $\pi(\WH{K}_{h})=\WH{K}_h^{\ell}$. Let also $L_e$ be the inclusion operator that maps $e$ into an edge of $K$ and let $L_{\WT{K}_h}(K) =\WT{K}_h$.
\begin{figure}[!hbt]
\begin{center}
\includegraphics[scale=0.9]{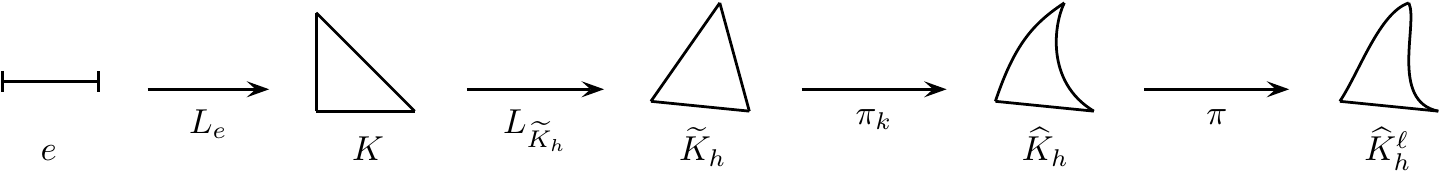}
\caption{Mappings used in the proof of Lemma \ref{Gamma2GammahSmall}.}\label{img:Lemma4_1}
\end{center}
\end{figure}
A tangent on an edge $\WH{e}_{h}\subset \WH{K}_h$ in $\Gamma_{h}^{k}$ is given by $\tau_{h} = \nabla(\pi_k\circ L_{\WT{K}_h}\circ L_e)$.
Analogously,  a tangent on the surface lifted edge $\WH{e}_{h}^{\ell}\subset \WH{K}_h^{\ell}$ in $\Gamma$ is given by $\tau = \nabla \pi \tau_{h} $.
We denote by $\overline{\tau}_{h}$ and $\overline{\tau}$ respectively the unit tangents of $\WH{e}_{h}$ and $\WH{e}_{h}^{\ell}$, and let $\lambda=\norm{\tau_{h}}_{l^2}$.
We will now prove estimate (\ref{5rel}). Let$\dx$ be the Lebesque measure on the reference interval $e$. We then have
\begin{align*}
\dshk&=\lambda \dx,\\
\ds & =\sqrt{\norm{(\nabla \pi \tau_{h})^T\cdot \nabla \pi \tau_{h}}_{l^2}} \dx =\lambda \sqrt{ \norm{(\nabla \pi \overline{\tau}_{h})^T\cdot \nabla \pi\overline{\tau}_{h}}_{l^2}} \dx = \underbrace{\norm{\nabla \pi \overline{\tau}_{h}}_{l^2}}_{\delta_{\WH{e}_{h}}}\dshk.
\end{align*}
Having characterised $\delta_{\WH{e}_{h}}$, we wish to show that
\[1 - Ch^{k+1}\leq\norm{\nabla \pi\overline{\tau}_{h}}_{l^2}\leq 1 + Ch^{k+1}.\]
Making use of (\ref{Da}) and (\ref{1rel}), we have that
\begin{equation}
\label{partialEstimateDa}
\norm{\nabla \pi \overline{\tau}_{h}}_{l^2}\leq\norm{\nabla \pi}_{l^2} \norm{\overline{\tau}_{h}}_{l^2}\leq\norm{P-dH}_{l^2}\leq 1+Ch^{k+1}.
\end{equation}
Next, to provide a lower bound for $\norm{\nabla \pi \overline{\tau}_{h}}_{l^2}$, we consider
$$\tau -\tau_{h}=(\nabla \pi-P_h)\tau_{h}=\lambda(\nabla \pi-P_h)\overline{\tau}_{h}.$$
Recalling the definition of the projection matrices $P$ and $P_h$, we have that
$$\norm{\tau -\tau_{h}}_{l^2} \leq \lambda \norm{(P-P_h)-dH}_{l^2} \norm{\overline{\tau}_{h}}_{l^2} \leq \lambda C h^{k}.$$
Using the reverse triangle inequality, we obtain
\begin{equation}\label{eq:lowerBoundDpi}
\lambda \norm{\nabla \pi\overline{\tau}_{h}}_{l^2}=\norm{\tau}_{l^2}\geq \norm{\tau_{h}}_{l^2} - \norm{\tau -\tau_{h}}_{l^2}\geq \lambda (1 - Ch^{k})
\end{equation}
and, dividing by $\lambda$ and using (\ref{partialEstimateDa}), we obtain the sub-optimal estimate
\begin{equation}
\label{estimateDa1}
1 - Ch^{k}\leq\norm{\nabla \pi\overline{\tau}_{h}}_{l^2}\leq 1 + Ch^{k+1}.
\end{equation}
The lower bound $\ref{estimateDa1}$ can be improved in an iterative way as follows. We consider
\begin{equation}\label{eq:lowerBound-1}
\lambda \norm{\nabla \pi\overline{\tau}_{h}}_{l^2}=\norm{\tau}_{l^2}\geq \norm{P\tau_{h}}_{l^2} - \norm{P\tau_{h} - \tau}_{l^2}.
\end{equation}
Then, using again the reverse triangular inequality, we have that
\begin{equation}\label{eq:lowerBound1}
\norm{P\tau_{h}}_{l^2}=\lambda\norm{P\overline{\tau}_{h}}_{l^2}\geq\lambda(\norm{\overline{\tau}}_{l^2}-\norm{\overline{\tau} -P\overline{\tau}_{h}}_{l^2})=\lambda(1-\norm{\overline{\tau} -P\overline{\tau}_{h}}_{l^2}).
\end{equation}
Since $\overline{\tau}, n,\nu$ form an orthonormal basis of $\mathbb{R}^3$ and recalling that $P$ maps vectors into the tangential space of $\Gamma$ (hence have null normal component), we get
\begin{align}
\lambda(1-\norm{\overline{\tau} -P\overline{\tau}_{h}}_{l^2}) &= \lambda(1-\norm{1-(\overline{\tau},P\overline{\tau}_{h})\overline{\tau} - (n,P\overline{\tau}_{h})n }_{l^2})\notag \\
&\geq \lambda(1-\norm{(1-(\overline{\tau},\overline{\tau}_{h}))}_{l^2}- \norm{(n,\overline{\tau}_{h}) }_{l^2})\notag \\
&\geq \lambda(1-\norm{\overline{\tau} -\overline{\tau}_{h}}_{l^2}^2- \norm{(n,\overline{\tau}_{h}) }_{l^2}). \label{estimateNorm}
\end{align}
Now 
$$\overline{\tau}_{h}-\overline{\tau} = (P_h-\frac{\nabla \pi}{\norm{\nabla \pi\overline{\tau}_{h}}_{l^2}})\overline{\tau}_{h},$$
so using (\ref{estimateDa1}) and a Taylor expansion argument, it is easy to see that 
\begin{equation}\label{eq:ineq_diff_unitary_tangent}
\norm{\overline{\tau}_{\WH{e}_{h}}-\overline{\tau}_{\WH{e}_{h}^{\ell}}}_{l^2}\lesssim  h^{k}.
\end{equation}
To deal with the last term of (\ref{estimateNorm}) we note that
\begin{equation*}
(n,\overline{\tau}_{h}) = (\overline{\tau} \times \nu,\overline{\tau}_{h})=(\nu,\overline{\tau}_{h} \times \overline{\tau})=(\nu,\overline{\tau}_{h} \times \frac{\nabla \pi\overline{\tau}_{h}}{\norm{\nabla \pi\overline{\tau}_{h}}_{l^2}}).
\end{equation*}
Then, using the sub-optimal lower bound (\ref{estimateDa1}) and a Taylor expansion argument, we get
\begin{equation*}
(\nu,\overline{\tau}_{h} \times \frac{\nabla \pi\overline{\tau}_{h}}{\norm{\nabla \pi\overline{\tau}_{h}}_{l^2}})=\frac{1}{\norm{\nabla \pi\overline{\tau}_{h}}_{l^2}}(\nu,\overline{\tau}_{h} \times \nabla \pi\overline{\tau}_{h})\lesssim (\nu,\overline{\tau}_{h} \times \nabla \pi\overline{\tau}_{h}).
\end{equation*}
Using the definition of $P$ and (\ref{Da}), we have that
\begin{equation}
\label{DaTbarIneq}
\nabla \pi\overline{\tau}_{h}=(P-dH)\overline{\tau}_{h}=\overline{\tau}_{h}-(\nu\cdot \overline{\tau}_{h})\nu-dH\overline{\tau}_{h}.
\end{equation}
Now, using (\ref{DaTbarIneq}), we can write
\begin{align*}
&(\nu,\overline{\tau}_{h} \times \nabla \pi\overline{\tau}_{h})=\bigg(\nu,\overline{\tau}_{h} \times (\overline{\tau}_{h} - (\overline{\tau}_{h}\cdot \nu)\nu - dH\overline{\tau}_{h})\bigg)=-(\nu,\overline{\tau}_{h} \times dH\overline{\tau}_{h}).
\end{align*}
Hence,
\begin{equation}\label{estimateNormN}
\norm{(n,\overline{\tau}_{h}) }_{l^2}\lesssim \norm{d}_{L^\infty}\norm{(\nu,\overline{\tau}_{h} \times H\overline{\tau}_{h})}_{l^2}\lesssim h^{k+1}.
\end{equation}
Combining (\ref{estimateNormN}) and (\ref{eq:ineq_diff_unitary_tangent}) with (\ref{estimateNorm}) we obtain that
\begin{equation}\label{eq:lowerBound-2}
\norm{P\tau_{h}}_{l^2}\geq\lambda(1-\norm{(1-(\overline{\tau},P\overline{\tau}_{h}))\overline{\tau} - (n,P\overline{\tau}_{h})n }_{l^2})\geq\lambda(1-Ch^{k+1}).
\end{equation}
For the second term in the right-hand side of (\ref{eq:lowerBound-1}), notice that
\begin{equation}\label{eq:lowerBound3}
\norm{\tau -P \tau_{h}}_{l^2}= \norm{\nabla \pi \tau_{h} -P \tau_{h}}_{l^2} = \norm{dH\tau_{h}}_{l^2}\leq \lambda C h^{k+1}.
\end{equation}
We are now ready to improve the lower bound in $\ref{estimateDa1}$. By making use of (\ref{eq:lowerBound3}) and (\ref{eq:lowerBound-2}) in (\ref{eq:lowerBound-1}), we get
\begin{equation}\label{eq:lowerBoundTau}
\norm{\nabla \pi\overline{\tau}_{h}}_{l^2}\geq 1 - Ch^{k+1}
\end{equation}
which proves (\ref{5rel}).

To prove (\ref{7rel}), we need to preliminary prove the following auxiliary inequalities:
\begin{equation} \label{eq:ineq_tangent_conormal}
|(\overline{\tau} , n_{h})| \lesssim  h^{k+1},
\end{equation}
\begin{equation} \label{eq:ineq_conormal_conormal}
|1 - (n , n_{h})| \lesssim  h^{2k}.
\end{equation}
We start showing (\ref{eq:ineq_tangent_conormal}). Using the property of the cross product, we get
\begin{equation}
\label{CrossProductIneq}
(\overline{\tau} , n_{h})=(\overline{\tau} , \nu_h \times \overline{\tau}_{h})=(\nu_h, \overline{\tau}_{h} \times \overline{\tau})=(\nu_h, \overline{\tau}_{h} \times \nabla \pi\overline{\tau}_{h}).
\end{equation}
Replacing (\ref{DaTbarIneq}) in (\ref{CrossProductIneq}), we obtain
$$(\overline{\tau} , n_{h})=[\nu\cdot (\overline{\tau}_{h}-\overline{\tau})](\overline{\tau}_{h},\nu \times \nu_h)-(\nu_h,\overline{\tau}_{h}\times dH\overline{\tau}_{h}).$$
Taking the absolute value and using (\ref{1rel}), (\ref{3rel}) and (\ref{eq:ineq_diff_unitary_tangent}), we find
$$|(\overline{\tau} , n_{h})|\lesssim  h^{2k+1}+C h^{k+1}\lesssim  h^{k+1}.$$
In order to prove (\ref{eq:ineq_conormal_conormal}), we start showing that the following holds
\begin{equation}\label{eq:ineq_nu_conormal}
|(\nu,n_{h})|\lesssim  h^k.
\end{equation}
Indeed, using again the properties of the cross and scalar products, we obtain:
$$|(\nu,n_{h})|=|(\nu,\nu_h\times\overline{\tau}_{h})|= |(\nu_h,\overline{\tau}_{h}\times \nu)|= |(\nu_h,\overline{\tau}_{h}\times (\nu-\nu_h))|\lesssim  h^k.$$
Since the vector $n_{h}$ is of unit length, there exist $a(x),b(x),c(x)\in\mathbb{R}$ satisfying\\ $a^2+b^2+c^2=1$ such that
$$n_{h}=a \overline{\tau} + b n + c \nu,$$
where $a=(\overline{\tau},n_{h})$, $b=(n,n_{h})$ and $c=(\nu,n_{h})$.
Hence, using (\ref{eq:ineq_tangent_conormal}), (\ref{eq:ineq_nu_conormal}) and a Taylor expansion argument, we get
$$b=\pm\sqrt{1-a^2-c^2}=\pm\sqrt{1+O(h^{2k})}=\pm1+O(h^{2k}).$$
The inequality (\ref{eq:ineq_conormal_conormal}) follows by assuming that the mesh size $h$ of $\WH{\calT}_h$ is chosen small enough so that $b = 1+O(h^{2k})$. Finally, writing $P n_{h} = (\overline{\tau},P n_{h})\overline{\tau} + (n,P n_{h}) n$, we obtain (\ref{7rel}), i.e.,
\begin{align*}
\norm{ n -P n_{h}}_{L^{\infty}(\WH{e}_{h})} 
&= \norm{n - (\overline{\tau},P n_{h})\overline{\tau} + (n,P n_{h})n}_{L^{\infty}(\WH{e}_{h})} \\ 
& \leq |1 - (n,P n_{h})| + |(\overline{\tau},P n_{h})| \\
&= |1 - ( n,n_{h})| + |(\overline{\tau},n_{h})| = O(h^{k+1}).
\end{align*}\end{proof}

\bibliography{SINUM-refs}

\end{document}